\documentclass[a4paper,12pt]				{amsart}
\usepackage[T1]								{fontenc} 							
\usepackage[ansinew]						{inputenc} 							
\usepackage									{ellipsis} 							
\usepackage									{lmodern} 							
\usepackage									{amssymb, amsmath, amsthm}			
\usepackage									{bm, bbm}							
\usepackage[fixamsmath]						{mathtools} 						
\usepackage[shortlabels]					{enumitem} 							
\usepackage[spacing,kerning]				{microtype}
\usepackage									{geometry}
\usepackage 								{hyperref}
\usepackage[style=alphabetic]				{biblatex}
\usepackage									{csquotes}							
\usepackage									{todonotes}
\usepackage[foot]							{amsaddr}
\usepackage									{subcaption}

\allowdisplaybreaks
\microtypecontext{spacing=nonfrench}

\theoremstyle{plain}
\newtheorem{thm}					{Theorem}[section]

\newtheorem{cor}			[thm]	{Corollary}
\newtheorem{proposition}	[thm]	{Proposition}

\theoremstyle{definition}

\theoremstyle{remark}

\newtheorem*{exa*}					{Example}
\newtheorem*{rem*}					{Remark}

\DeclareMathOperator	{\IN}		{\mathbb{N}} 	
	\DeclareMathOperator	{\IR}		{\mathbb{R}}
	\DeclareMathOperator	{\IE}		{\mathbb{E}}

\DeclareMathOperator	{\Var}		{Var}

\DeclareMathOperator	{\Ent}		{Ent}
\DeclareMathOperator	{\Cov}		{Cov}

\DeclareMathOperator	{\supp}		{supp}
\DeclareMathOperator	{\bb}		{\boldsymbol{\beta}}
\DeclareMathOperator	{\ERGM}		{ERGM}

\DeclarePairedDelimiter	\abs		{\lvert}	{\rvert}
\DeclarePairedDelimiter	\norm		{\lVert}	{\rVert}
\DeclarePairedDelimiter	\skal		{\langle}	{\rangle}

\newcommand				{\eins}		{\text{$\mathbbm{1}$}}

\renewcommand 			{\epsilon}	{\varepsilon}
\renewcommand			{\phi}		{\varphi}
\renewcommand			{\partial}	{\mathfrak{d}}
\renewcommand			{\tilde}	{\widetilde}

\numberwithin			{equation}{section}
\title{Logarithmic Sobolev inequalities for finite spin systems and applications}

\author		{Holger Sambale$^1$ and Arthur Sinulis$^1$}
\address	{$^1$Faculty of Mathematics, Bielefeld University, Bielefeld, Germany}

\email[A1]		{hsambale@math.uni-bielefeld.de}
\email[A2]		{asinulis@math.uni-bielefeld.de}

\thanks{This research was supported by the German Research Council (DFG) via the CRC 1283 \emph{Taming uncertainty and profiting from randomness and low regularity in analysis, stochastics and their applications}. }
\date{\today}

\begin{document}
\begin{abstract}
	We derive sufficient conditions for a probability measure on a finite product space (a \emph{spin system}) to satisfy a (modified) logarithmic Sobolev inequality. We establish these conditions for various examples, such as the (vertex-weighted) exponential random graph model, the random coloring and the hard-core model with fugacity. 
    
    This leads to two separate branches of applications. The first branch is given by mixing time estimates of the Glauber dynamics. The proofs do not rely on coupling arguments, but instead use functional inequalities. As a byproduct, this also yields exponential decay of the relative entropy along the Glauber semigroup.
    Secondly, we investigate the concentration of measure phenomenon (particularly of higher order) for these spin systems. We show the effect of better concentration properties by centering not around the mean, but a stochastic term in the exponential random graph model. From there, one can deduce a central limit theorem for the number of triangles from the CLT of the edge count. In the Erd{\"o}s--R{\'e}nyi model the first order approximation leads to a quantification and a proof of a central limit theorem for subgraph counts.
\end{abstract}

\maketitle

\section{Introduction}                  \label{section:Introduction}
\emph{Spin systems} are ubiquitous in the modeling of various phenomena, ranging from toy models to explain ferromagnetism (the Ising and the Potts model, or more generally the random cluster model) to voter models (e.g. interpreting the Ising model as a social choice with binary options), various network models (such as the Erd\"os-Renyi or the exponential random graph model) and models with \emph{hard constraints} such as the random proper coloring model or the hard-core model. 

From the physical viewpoint, a spin system models a collection of particles attaining different states and interacting with each other, so that the complete system consists of a set of \emph{configurations} of the form $\mathcal{X}^{\mathcal{I}}$. 
Mathematically, a spin system can be described as a probability measure $\mu$ on such a product space $\mathcal{Y} \coloneqq \mathcal{X}^{\mathcal{I}}$, and hard constraints translate into conditions on the support of the probability measure. Here we consider finite spin systems, i.\,e. the sets $\mathcal{X}$ (the \emph{spins}) and $\mathcal{I}$ (the \emph{sites}) are finite. 

Albeit very elementary, these finite spin systems can have a rich dependence structure among the sites. Indeed, many toy models of statistical mechanics are defined as finite spin systems (as the aforementioned Ising model on a finite graph). 
We are interested in the regimes in which the sites exhibit behavior typical of independent random variables. To this end, we define suitable notions of \emph{weak dependence} which, on the technical side, lead to (modified) logarithmic Sobolev inequalities.

Based on these general results, we study two different branches of applications. Firstly we want to study the mixing time of the \emph{Glauber dynamics} associated to a sequence of weakly dependent spin systems $(\mu_n)_n$ on $\mathcal{X}^{\mathcal{I}_n}$ which are rapidly mixing, i.\,e. the mixing time is $O(\abs{\mathcal{I}_n} \log \abs{\mathcal{I}_n})$. Secondly, we study the concentration of measure phenomenon (of higher order) for these kinds of models. From the very definition of both concepts, it is necessary to consider a sequence of finite spin systems with increasing number of sites.

\subsection{The models}	\label{subsection:themodels}
Our main results are valid for arbitrary weakly dependent spin systems. At the same time, throughout this paper we put a special focus on a number of models for which we establish sufficient conditions for weak dependence and apply our general results. Let us emphasize that any of the models depends on a parameter $n \in \IN$, so that we are considering a sequence of spin systems with a growing number of sites. We will usually suppress this dependence on the number of sites. We sometimes specify a spin system by defining a Hamiltonian, i.\,e. a function $H: \mathcal{Y} \to \IR$. The spin system associated to $H$ is given by the \emph{Gibbs measure}
\[
\mu(\sigma) = \mu_H(\sigma) = Z^{-1} \exp(H(\sigma)) \quad \text{for} \quad Z = \sum_{\sigma \in \mathcal{Y}} \exp(H(\sigma)).
\]

\subsubsection{The exponential random graph model}
The first spin system we consider is the exponential random graph model. 
For a thorough historical overview and asymptotic results we refer to the well-written survey \cite{Cha16}. The model is an extension of the Erd{\"o}s--R{\'e}nyi model to account for dependence between the edges.

For two simple graphs $G_1 = (V_1, E_1)$ and $G_2 = (V_2, E_2)$ let $N_{G_1}(G_2)$ be the number of graph homomorphisms from $G_1$ to $G_2$, i.\,e. injective maps $\phi: V_1 \to V_2$ which preserve edges. Moreover $\mathcal{G}_n$ shall be the set of all simple graphs on $n$ vertices, labeled $\{1,\ldots, n\}$, and $\mathcal{I}_n \coloneqq \{ (i,j) \in \{1,\ldots, n\}^2 : i < j \}$.
Let $\bb = (\beta_1,\ldots, \beta_s) \in \IR^s$ and $G_1, \ldots, G_s$ be simple, connected graphs $G_i = (V_i, E_i)$. The exponential random graph model with parameters $(\boldsymbol{\beta}, G_1, \ldots, G_s)$, denoted by $\mu_{\bb}$, is defined as the spin system on $\{0,1 \}^{\mathcal{I}_n}$ associated to the Hamiltonian
\[
	H_{\boldsymbol{\beta}}(x) \coloneqq n^2 \sum_{i = 1}^s \beta_i \frac{N_{G_i}(x)}{n^{\abs{V_i}}}.
\]
We frequently abbreviate this as $\ERGM(\boldsymbol{\beta},G_1,\ldots,G_s)$. By convention, $G_1$ is the complete graph on two vertices $K_2$. Note that for $s = 1$ we obtain the Erd{\"o}s--R{\'e}nyi model with parameter $p = e^\beta (1+e^\beta)^{-1}$.

For any set of parameters $(\bb, G_1, \ldots, G_s)$ we define the functions $\Phi_{\bb}, \phi_{\bb}: [0,1] \to \IR$
\begin{align*}
  \Phi_{\bb} (x) & = \sum_{i = 1}^s \beta_i \abs{E_i} x^{\abs{E_i}-1} = \beta_1 + \sum_{i = 2}^s \beta_2 \abs{E_i} x^{\abs{E_i} -1} \\
  \phi_{\bb}(x)  & = \frac{\exp(2\Phi_{\bb}(x))}{1+\exp(2\Phi_{\bb}(x))} = \frac{1}{2}(1 + \tanh(\Phi_{\bb}(x))),
\end{align*}
and set $\abs{\bb} \coloneqq (\abs{\beta_1}, \ldots, \abs{\beta_s})$. The function $\Phi'_{\abs{\bb}}$ (i.\,e. the derivative of $\Phi_{\abs{\bb}}$) will appear in the condition of weak dependence in this model.

\subsubsection{Vertex-weighted random graph model}
A second model of random graphs - the vertex-weighted exponential random graph model - was recently introduced in \cite{DEY17}. The parameter-space is three-dimensional, i.\,e. $\beta = (\beta_1, \beta_2, p) \in \IR^2 \times (0,1)$, and the model is the spin system on $\mathcal{Y} = \{0,1\}^n$ defined via the Hamiltonian
\[
H(\sigma) \coloneqq \log\left( \frac{p}{1-p} \right) \sum_i \sigma_i + \frac{\beta_1}{n}\sum_{i \neq j} \sigma_i \sigma_j + \frac{\beta_2}{n^2}\sum_{i \neq j \neq k} \sigma_i \sigma_j \sigma_k.
\]
Note that it resembles the Hamiltonian in the exponential random graph model. On the other hand, it can also be seen as an extension of the Curie--Weiss model on the complete graph with interactions given by a quadratic \emph{and} a cubic form. 
We define the function
\[
	\phi_\beta(x) \coloneqq \frac{\exp(h_\beta(x))}{1+\exp_\beta(h(x))} = \frac{\exp\left( \beta_1 x + \beta_2 x^2 + \log(p/(1-p)) \right)}{1+\exp\left( \beta_1 x + \beta_2 x^2 + \log(p/(1-p)) \right)}.
\]
Similarly to the ERGM, the derivative of $\phi_{\beta}$ will determine whether the system is weakly dependent.

\subsubsection{Random coloring model}
Given a finite graph $G = (V,E)$ and a set of colors $C = \{1,\ldots,k\}$, the configuration space in the random coloring model is the set of all proper colorings $\Omega_0 \subset C^V$, i.e. the set of all $\sigma \in C^V$ such that $\{v,w \} \in E \Rightarrow \phi_v \neq \phi_w$, and $\mu = \mu(G,C)$ denotes the uniform distribution on $\Omega_0$. 

\subsubsection{Hard-core model with fugacity}
Another model with hard constraints is the hard-core model with \emph{fugacity parameter} $\lambda$. Given a graph $G = (V,E)$, the hard-core model $\mu$ is the spin system on $\mathcal{Y} = \{0,1\}^{V}$ satisfying
\[
\mu(\sigma) = \begin{cases}
Z^{-1} \prod_i \lambda^{\sigma_i} & \sigma \text{ admissible} \\
0 & \text{otherwise}
\end{cases}.
\] 
Here, an admissible configuration satisfies $\sigma_v \sigma_w = 0$ for all $\{v, w\} \in E$.

\subsection{Outline}	\label{subsection:outline}
The structure of this paper is as follows. We formulate the main results for mixing times in Section \ref{section:mixingtimes} (see Theorems \ref{theorem:GlauberDynamicsRapidMixing} and \ref{theorem:AbstractResult}) and prove them in Section \ref{section:proofMixingTimes}. The concentration of measure results are given in Section \ref{section:CoM}, with Subsection \ref{subsection:triangleERGM} containing the concentration inequalities for the triangle count in the ERGM (Theorem \ref{theorem:ERGMtriangleLpandTails}) as well as the central limit theorem results (Corollary \ref{corollary:CLTERGM}). Section \ref{subsection:CLTER} discusses the central limit theorem in the Erd{\"o}s--R{\'e}nyi model. These results are applications of Theorems \ref{theorem:com} and \ref{theorem:fdPolynomials} given in Subsection \ref{subsection:CoMgeneralresults}. We prove all concentration results in Section \ref{section:ProofsCoM}. The proof of the (modified) logarithmic Sobolev inequalities in the models from Subsection \ref{subsection:themodels} can be found in Section \ref{section:weakdependence}.

\section{Mixing times}	\label{section:mixingtimes}
For spin systems $\mu$ with many sites, one way to sample from these is to use the Markov chain convergence theorem. One canonical Markov Chain is the Glauber dynamics, which is a $\mathcal{Y}$-valued ergodic Markov chain $(Y_t)_{t \in \IN_0}$ with reversible distribution $\mu$. At each step, it selects a site $i \in \mathcal{I}$ uniformly at random and updates it with the conditional probability given $\overline{x}_i = (x_j)_{j\ne i}$, i.\,e. its transition probability is given by
\[
	P(x,y) = \abs{\mathcal{I}}^{-1} \sum_{i \in \mathcal{I}} \mu(y_i \mid \overline{x}_i) \prod_{j \neq i} \delta_{x_j = y_j},
\]
where $\delta$ is the Dirac delta. Under mild conditions, if $(Y_t)_{t \in \IN_0}$ is a Markov chain on a finite space $\mathcal{Y}$ with a reversible measure $\nu$, the distribution of $Y_t$ converges to $\nu$. 
This convergence can be quantified using various metrics between probability measures, and we choose the total variation distance
	$d_{\mathrm{TV}}(\mu_1, \mu_2) \coloneqq \sup_{A \subset \mathcal{Y}} \abs{\mu_1(A) - \mu_2(A)} = \frac{1}{2} \sum_{x \in \mathcal{Y}} \abs{\mu_1(x) - \mu_2(x)}.$
In the continuous-time case, we define the \emph{mixing time} as
\begin{align*}
	t_{\mathrm{mix}} \coloneqq \inf \{ t \in \IR_+ : \max_{y \in \mathcal{Y}} d_{\mathrm{TV}}(P^t(y, \cdot), \nu) \le e^{-1} \}.
\end{align*}
Here, $P^t(y,\cdot)$ is defined as the distribution of $Y_t$ given $Y_0 = y$. 

Our proofs rely on modified logarithmic Sobolev inequalities. Let $P$ be the transition matrix of a Markov chain on $\mathcal{Y}$ and $-L = I - P$ be its generator. If $P$ is reversible with respect to a measure $\mu$, we can define the entropy functional 
	$\Ent_\mu(f) \coloneqq \IE_\mu f \log f - \IE_\mu f \log( \IE_\mu f ) \text{ for } f \ge 0$
and the Dirichlet form
	$\mathcal{E}(f,g) \coloneqq \IE_\mu f (-Lg).$
We say that the triple $(\mathcal{Y}, P, \mu)$ (or in short $P$, if the space and the measure are clear from the context) satisfies a \emph{modified logarithmic Sobolev inequality} with constant $\rho_0$ (in short: $\mathrm{mLSI}(\rho_0)$), if for all $f: \mathcal{Y} \to \IR_+$ we have
\begin{align}	\label{eqn:defi:modLSI}
	\Ent_\mu(f) \le \frac{\rho_0}{2} \mathcal{E}(f,\log f).
\end{align}
The smallest $\rho_0$ in \eqref{eqn:defi:modLSI} is called modified logarithmic Sobolev (or entropy) constant, see e.g. \cite{BT06} and the definition of $\beta$ in \cite{GQ03}. It is known that the modified logarithmic Sobolev constant can be used to bound the mixing time for the total variation distance of (the distribution of) a Markov semigroup and its trend to equilibrium, see for example \cite[Theorem 2.4]{BT06}.

We will prove the following theorem, which establishes the $\mathrm{mLSI}(\rho_0)$ for the models introduced in Section \ref{subsection:themodels}, and consequently the exponential decay of the entropy along the Glauber semigroup and the rapid mixing property thereof.

\begin{thm}\label{theorem:GlauberDynamicsRapidMixing}
The Glauber dynamics satisfies a $\mathrm{mLSI}(\rho_0)$ with $\rho_0 \le C \abs{\mathcal{I}_n}$ in all following spin systems. The constant $C$ may depend on the parameters of the spin system, but not on $n$.
\begin{enumerate}
\item Any exponential random graph model $\mu_{\bb}$ such that $\frac{1}{2} \Phi'_{\abs{\bb}}(1) < 1$.
\item The vertex-weighted ERGM for $\beta \coloneqq (\beta_1, \beta_2, p)$ satisfying $\sup_{x \in (0,1)} \abs{\phi_\beta'(x)} < 1$.
\item The random coloring model $\mu_{G_n}$ for any sequence of graphs $G_n = (V_n,E_n)$ with uniformly bounded maximum degree $\Delta$ and $k \ge 2\Delta + 1$.
\item The hard-core model with fugacity $\lambda$ on any sequence of graphs $G_n = (V_n, E_n)$ with bounded maximum degree $\Delta$ and $\lambda < \frac{1}{\Delta-1}$.
\end{enumerate}
Consequently, the relative entropy decreases exponentially along the Glauber semigroup, i.\,e. for any $\mu$-density $f$ and the Glauber semigroup $(P_t)_{t \ge 0}$ we have 
\[
\Ent_{\mu}(P_t f) \le \exp(-2t/\rho_0) \Ent_{\mu}(f).
\]
Finally, in all the cases, the Glauber dynamics is rapidly mixing.
\end{thm}

Let us put the results of Theorem \ref{theorem:GlauberDynamicsRapidMixing} into context.

Concerning the exponential random graph model, we are sure that the condition $\frac{1}{2} \Phi_{\abs{\beta}}'(1) < 1$ is not optimal. For $\beta_2, \ldots, \beta_s \ge 0$ is was proven in \cite[Theorem 5]{BBS11} that the Glauber dynamics is rapidly mixing whenever there is only one solution $a^*$ to the equation $\phi_{\bb}(a) = a$ satisfying $\phi_{\bb}'(a^*) < 1$, which is implied by our condition. 
On the other hand, Theorem \ref{theorem:GlauberDynamicsRapidMixing} implies exponential decay of the relative entropy along the Glauber semigroup and also applies to negative parameters. 
The assumption $\frac{1}{2} \Phi'_{\abs{\beta}}(1) < 1$ is also present in \cite[Theorem 6.2]{CD13}. 

An easy application shows that for the ERGM with $s = 2$ and a simple, connected graph $G_2$ on $e_2$ edges, the condition $\abs{\beta_2} < \binom{e_2}{2}^{-1}$ guarantees rapid mixing. 

Rapid mixing for the vertex-weighted exponential random graph model has been established in \cite[Theorem 7]{DEY17} under a similar condition as in \cite{BBS11}.

The Glauber dynamics for the random coloring model on a sequence of bounded-degree graphs was shown to be rapidly mixing in \cite{Jer95} for $k \ge 2\Delta + 1$ via a path coupling approach. We recover these results using the entropy method. Again, this yields exponential decay of the relative entropy along the Glauber semigroup.

It was shown in \cite{Vi01} that if $G_n = (V_n, E_n)$ is a sequence of graphs with uniformly bounded degree $\Delta$ and $\lambda < \frac{2}{\Delta-2}$, then the Glauber dynamics associated to the hard-core model with fugacity $\lambda$ is rapidly mixing. Interestingly, with methods closer to the Bakry-Emery theory and a characterization of Ricci curvature for Markov chains, \cite{EHMT17} have shown that the hard-core model with fugacity has a positive Ricci curvature under the assumption $\lambda \le \frac{1}{\Delta}$, which also implies a $\mathrm{mLSI}(\rho_0)$.

Theorem \ref{theorem:GlauberDynamicsRapidMixing} is itself an application of the following general theorem. The (rather technical) quantity $\tilde{\beta}$ and the notion of the interdependence matrix $J$ will be defined in Section \ref{section:weakdependence}, see equations \eqref{eqn:DefiInterdependence} and \eqref{eqn:DefiBetaTilde}.

\begin{thm}					\label{theorem:AbstractResult}
	Let $\mu$ be a spin system on $\mathcal{Y} \coloneqq \mathcal{X}^{\mathcal{I}}$ for finite sets $\mathcal{X}$ and $\mathcal{I}$. Assume that for some constants $\alpha_1, \alpha_2 > 0$ we have the lower bound on the conditional probabilities
	\begin{align}	\label{eqn:LowerBoundConditionalProbability}
		\tilde{\beta}(\mu) \ge \alpha_1	
	\end{align}
	and an upper bound on some interdependence matrix $J$
	\begin{align}	\label{eqn:UpperBoundInterdependenceMatrix}
		\norm{J}_{2 \to 2} \le 1 - \alpha_2.				
	\end{align}
	The Glauber dynamics associated to $\mu$ satisfies an $\mathrm{mLSI}(2 \abs{\mathcal{I}}\alpha_1^{-1} \alpha_2^{-2})$.
	As a consequence, given any $\mu$-density $f$, the density $f_t = P_t f$ ($(P_t)_{t \ge 0}$ is the Glauber semigroup) satisfies
	\begin{equation}	\label{eqn:ExponentialDecayRelativeEntropy}
		\mathrm{Ent}_\mu(P_t f) \le \mathrm{Ent}_\mu(f) \exp\left( - \frac{\alpha_1 \alpha_2^2}{2\abs{\mathcal{I}}} t \right).
	\end{equation}
	
	Furthermore, if $(\mu_n)_n$ is a sequence of spin systems on $\mathcal{Y}_n \coloneqq \mathcal{X}^{\mathcal{I}_n}$ (with $\abs{\mathcal{I}_n} \to \infty$) satisfying \eqref{eqn:LowerBoundConditionalProbability} and \eqref{eqn:UpperBoundInterdependenceMatrix} for some $n$-independent constants $\alpha_1, \alpha_2$, then the sequence of Glauber dynamics is rapidly mixing, i.e. $t_{\mathrm{mix}} = O(\abs{\mathcal{I}_n} \log \abs{\mathcal{I}_n})$.
\end{thm}

Clearly, the second part of Theorem \ref{theorem:AbstractResult} is of interest only for $\abs{\mathcal{I}_n} \to \infty$ as $n \to \infty$. In the case of spin systems without hard constraints, we can rephrase the conditions \eqref{eqn:LowerBoundConditionalProbability}, \eqref{eqn:UpperBoundInterdependenceMatrix}. Here, we define $I(\mu) \coloneqq \min_{i \in \mathcal{I}} \min_{y \in \mathcal{Y}} \mu( y_i \mid \overline{y}_i)$ as the minimal conditional probability.

\begin{cor}			\label{corollary:RapidMixingGibbsMeasure}
Let $(\mu_n)_n$ be a sequence of Gibbs measures on configuration spaces $\mathcal{Y}_n$ induced by Hamiltonians $H_n: \mathcal{Y}_n \to \IR$. If $I(\mu_n) \ge \alpha_1$ and $\norm{J_n}_{2 \to 2} \le 1-\alpha_2$ for some $\alpha_1 \in (0,1), \alpha_2 \in (0,1)$ and interdependence matrices $J_n$, then the (sequence of) Glauber dynamics associated to $\mu_n$ is rapidly mixing.
\end{cor}

\section{Concentration of measure}	\label{section:CoM}
Informally, as stated in \cite{Tal96b}, the \emph{concentration of measure phenomenon} can be described as the phenomenon that a function of $n$ i.i.d. random variables $X_1, \ldots, X_n$ tends to be very close to a deterministic quantity (e.g. its expected value or median), if it is not too sensitive to one of its parameters. The function is usually assumed to be Lipschitz continuous in some sense, depending on a suitably adapted notion of a gradient. In other words, the distribution of any Lipschitz function of independent random variables shows strong (more precisely: sub-Gaussian) concentration properties. For an introduction to the concentration of measure phenomenon and functional inequalities we refer to  the two monographs \cite{Led01} and \cite{BLM13} and the lecture notes \cite{vH16}.

It is also known that the restriction to Lipschitz functions is not necessary, if one aims at (sub-)exponential tails. Already the early works of \cite{Bo68, Bo70} prove $L^p$ norm estimates for polynomials of degree $d$ in Rademacher random variables which grow like $p^{d/2}$ (which can be translated into tail estimates). In the setting of independent sub-Gaussian random variables the Hanson--Wright inequality (see \cite{HW71, W73, RV13}) gives estimates for quadratic forms. In these cases, exponential tail decay holds even though the Lipschitz condition is not satisfied, which might be regarded as an extension of the concentration of measure phenomenon beyond the setting of Lipschitz-type functions. This idea has been developed in many different works, such as \cite{Vu02, BBLM05, La06, Ad06, SS12a, Wo13, AW15} among others. 

The concentration of measure results in this article employ the entropy method. To recall some notions, let $(\mathcal{Y}, \mathcal{A},\mu)$ be a probability space. An operator $\Gamma: L^\infty(\mu) \to L^\infty(\mu)$ is called a \emph{difference operator} if $\abs{\Gamma(af+b)} = a \abs{\Gamma(f)}$ for all $a > 0, b \in \IR$. We say that $\mu$ satisfies a logarithmic Sobolev inequality with respect to $\Gamma$ (or in short: a $\Gamma\textrm{-LSI}(\sigma^2)$), if for all bounded and measurable functions $f: \mathcal{Y} \to \IR$ we have
\begin{align}    \label{eqn:generalLSI}
  \Ent_\mu(f^2) \le 2 \sigma^2 \int \Gamma(f)^2 d\mu.
\end{align}
Here, $\Ent_\mu(f)$ is the entropy functional. The smallest $\sigma^2 > 0$ such that \eqref{eqn:generalLSI} holds is known as the logarithmic Sobolev constant.
From the two properties of a difference operator one can infer that a $\Gamma\textrm{-LSI}(\sigma^2)$ implies a \emph{Poincar{\'e} inequality}
\begin{align}    \label{eqn:generalPI}
  \Var_{\mu}(f) \le \sigma^2 \int \Gamma(f)^2 d\mu.
\end{align}
To any function $f$ and $i \in \mathcal{I}$ we associate the ``local variance'' in the $i$-th coordinate
\begin{align}
  \partial_i f(x)^2 =
  \begin{cases}
    \frac{1}{2} \iint \left( f(\overline{x}_i, y) - f(\overline{x}_i, y') \right)^2 d\mu(y \mid \overline{x}_i) d\mu(y' \mid \overline{x}_i) & \text{for } \overline{\mu}_i(\overline{x}_i) > 0 \\
    0                                                                                                                                        & \text{otherwise}.
  \end{cases}
\end{align}
Here, $\overline{x}_i = (x_j)_{j \in \mathcal{I}\backslash \{i\}}$ is a generic vector in $\mathcal{X}^{\mathcal{I} \backslash \{i\}}$, $\mu(\cdot \mid \overline{x}_i)$ denotes the conditional probability interpreted as a measure on $\mathcal{X}$, and $\overline{\mu}_i$ is the marginal measure on $\mathcal{X}^{\mathcal{I} \backslash \{i \}}$. More generally, for any $S \subset \mathcal{I}$ we write $\overline{x}_S, \mu(\cdot \mid \overline{x}_S), \overline{\mu}_S$ for the obvious analogues.

One type of difference operator is given by $\abs{\partial f} = (\sum_{i \in \mathcal{I}} (\partial_i f)^2)^{1/2}$, where $\abs{\cdot}$ is be the Euclidean norm of a vector. Indeed, if \eqref{eqn:generalLSI} holds for $\Gamma(f) = \abs{\partial f}$, we say that $\mu$ satisfies a $\partial\textrm{-LSI}(\sigma^2)$. A second type of difference operator is given by $\abs{\mathfrak{h}f} = \left( \sum_{i \in \mathcal{I}} (\mathfrak{h}_i f)^2\right)^{1/2}$
for
\begin{align}
  \mathfrak{h}_i f(x) = \norm{f(\overline{x}_i,y) - f(\overline{x}_i, y')}_{L^\infty(\mu(\overline{x}_i, \cdot) \otimes \mu(\overline{x}_i,\cdot))}.
\end{align}
It is easy to see that if $\mu$ satisfies a $\partial$-$\mathrm{LSI}(\sigma^2)$, then it also satisfies an $\mathfrak{h}$-$\mathrm{LSI}(\sigma^2/2)$.

As a short remark, let us note that the definition of a $\partial\mathrm{-LSI}(\sigma^2)$ is consistent with the definition using the Dirichlet form as in Section \ref{section:mixingtimes}.

Similarly to Theorem \ref{theorem:GlauberDynamicsRapidMixing}, the following theorem provides $\partial$-LSIs for all the models from Section \ref{subsection:themodels}.

\begin{thm}\label{theorem:LSIforAllModels}
For the following models the spin system $\mu$ satisfies a $\partial\mathrm{-LSI}(\sigma^2)$, where $\sigma^2$ may depend on the parameters of each model but not on $n$.
\begin{enumerate}
\item Any exponential random graph model $\mu_{\bb}$ such that $\frac{1}{2} \Phi'_{\abs{\bb}}(1) < 1$.
\item The vertex-weighted ERGM for $\beta \coloneqq (\beta_1, \beta_2, p)$ satisfying $\sup_{x \in (0,1)} \abs{\phi_\beta'(x)} < 1$.
\item The random coloring model $\mu_{G_n}$ for any sequence of graphs $G_n = (V_n,E_n)$ with uniformly bounded maximum degree $\Delta$ and $k \ge 2\Delta + 1$.
\item The hard-core model with fugacity $\lambda$ on any sequence of graphs $G_n = (V_n, E_n)$ with bounded maximum degree $\Delta$ and $\lambda < \frac{1}{\Delta-1}$.
\end{enumerate}
\end{thm}

\subsection{Triangle counts in the exponential random graph model} \label{subsection:triangleERGM}
The question of the distribution of the number of triangles in the Erd{\"o}s--R{\'e}nyi model is quite classical and well-studied. Therefore, it is an interesting task to find analogous results for the exponential random graph model. Although the edges in this model are dependent, a weak dependence condition should suffice to expect a similar behavior as in the case of independent edges (for example, see the large deviation results in \cite{CD13}, which in certain cases implies that an exponential random graph model is indistinguishable from an Erd{\"o}s--R{\'e}nyi model in the limit). Although large deviation results are purely asymptotic, one can still hope for similar behavior concerning certain statistics for finite $n$. \par
Recall that the exponential random graph model is a spin system with sites $\mathcal{I}_n \coloneqq \{ (i,j) \in \{1,\ldots, n\}^2 : i < j \}$. We let $\binom{\mathcal{I}_n}{3}$ be the set of all possibilities of choosing three distinct edges and
\begin{align*}
  \mathcal{T}_n \coloneqq \left\lbrace \{ e,f,g \} \in \binom{\mathcal{I}_n}{3} : e,f,g \text{ form a triangle} \right\rbrace.
\end{align*}
We define the number of triangles
  $T_3(x) \coloneqq \sum_{\{e_1, e_2, e_3\} \in \mathcal{T}_n} x_{e_1} x_{e_2} x_{e_3}.$
Our first result are multilevel concentration inequalities for $T_3$ and a ``linear approximation'' thereof. Define $\mu_2 \coloneqq \IE_{\mu_{\bb}} x_e x_f$ (for some edges $e \neq f \in \mathcal{I}_n$, $e \cap f \neq \emptyset$) and $f_1 \coloneqq \sum_{e \in \mathcal{I}_n} (x_e - \IE_{\mu_{\bb}}(x_e))$. From the definition of the ERGM it is clear that $\mu_2$ is well-defined.

\begin{thm}              \label{theorem:ERGMtriangleLpandTails}
  Let $\mu_{\bb}$ be an ERGM satisfying a $\partial$-LSI$(\sigma^2)$. There exists a constant $C = C(\sigma^2) > 0$ such that for all $t > 0$ we have the multilevel concentration bounds
  \begin{align}  \label{eqn:T3fluctuations}
    &\mu_{\bb}(\abs{{T}_3 - \IE_{\mu_{\bb}}T_3} \ge t) \le 2 \exp \left( - \frac{1}{C} \min \left( \Big( \frac{t}{n^{3/2}} \Big)^{2/3}, \frac{t}{\mu_1 n^{3/2}}, \Big( \frac{t}{\mu_2 n^2} \Big)^2 \right) \right) \\
  \label{eqn:T3minusf1fluctuations}
    &\mu_{\bb} \bigg( \abs{T_3 - \IE_{\mu_{\bb}}T_3 - (n-2)\mu_2 f_1} \ge t \bigg) \le 2\exp \left( - \frac{1}{C} \min \left( \Big( \frac{t}{n^{3/2}}\Big)^{2/3}, \frac{t}{\mu_1 n^{3/2}} \right) \right).
  \end{align}
\end{thm}

It is interesting to note the effect of subtracting the random variable $(n-2) \mu_2 f_1$. As the variance of $T_3$ is of order $n^4$, a normalization by $n^{-2}$ is necessary to obtain a stable variance, and inequality \eqref{eqn:T3fluctuations} gives suitable tail estimates. However, the random variable $T_3 - \IE_{\mu_{\bb}} T_3 - (n-2) \mu_2 f_1$ concentrates on a narrower range, since the variance is of order $n^{3}$, and equation \eqref{eqn:T3minusf1fluctuations} yields stretched-exponential tails in this case. 
In the Erd{\"o}s--R{\'e}nyi model, a short calculation shows
\begin{align*}
  \Var\left( T_3 \right)                & = \binom{n}{3} p^3(1-p^3) + \frac{1}{2} n(n-1)(n-2)(n-3) p^5(1-p) \\
  \Var\left( T_3 - (n-2)p^2 f_1 \right) & = \binom{n}{3}p^3(1-p^3).
\end{align*}
To complement these observations, inspecting \eqref{eqn:T3fluctuations}, we see that the normalization $n^{-2}$ corresponds to the factor $n^{-4}$ in the Gaussian part, whereas the exponential and stretched-exponential part require a normalization of $n^{-3/2}$ only.

\begin{figure}[!ht]
  \centering
  \begin{subfigure}[c]{0.49\textwidth}
    \includegraphics[width=\textwidth]{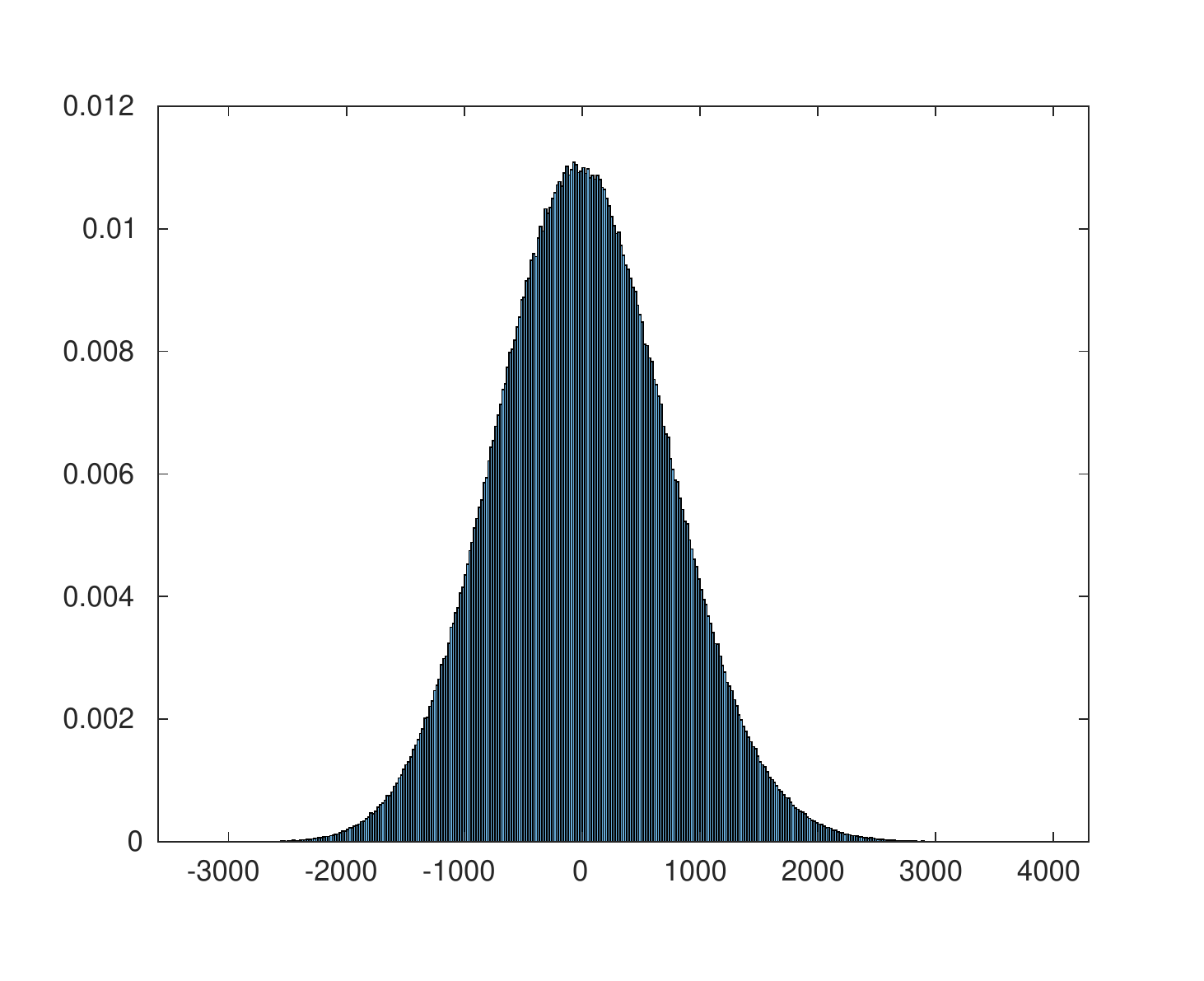}
  \end{subfigure}
  \begin{subfigure}[c]{0.49\textwidth}
    \includegraphics[width=\textwidth]{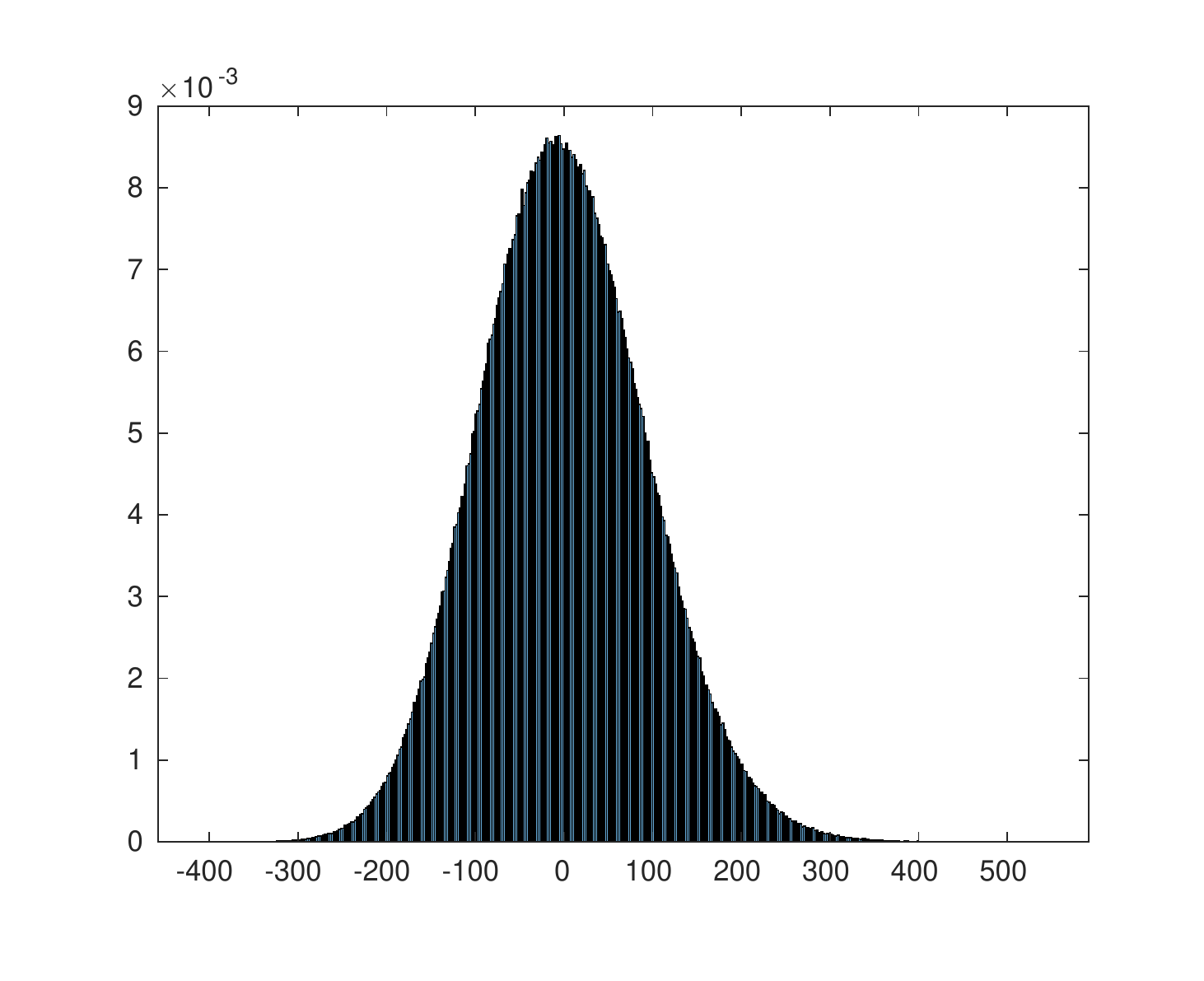}
  \end{subfigure}
  \caption{A comparison of the distributions of $T_3 - \mu_{\bb}(T_3)$ (left) and $T_3 - \mu_{\bb}(T_3) - (n-2)\mu_2 f_1$ (right) for $n = 100$, $\beta_1 = -0.1, \beta_2 = 0.05$ and $G_1 = K_2$ (an edge), $G_2 = K_3$ (a triangle) using the Glauber dynamics and roughly 2 million simulations.}
\end{figure}

The inequality \eqref{eqn:T3minusf1fluctuations} shows that $T_3$ fluctuates around the linear term $f_1$ on a lower order. This leads to the idea of mimicking the method of H{\'a}jek projection to deduce a central limit theorem for the triangle count from a CLT for the edge count. As far as we are aware, there are hardly any theoretical results on the distributional limits of the subgraph counts as $n \to \infty$ barring certain special cases. (One such example is the edge two-star model, which can also be interpreted as an Ising model, see \cite{Muk13a, Muk13b}.) 

\begin{cor}      \label{corollary:CLTERGM}
  Let $\mu_{\bb}$ be an ERGM satisfying a $\partial$-LSI$(\sigma^2)$. Assuming the central limit theorem $\binom{n}{2}^{-1/2} \sum_{e \in \mathcal{I}_n} (x_e - \IE_{\mu_{\bb}} x_e) \Rightarrow \mathcal{N}(0,v^2)$, we can infer
  \[
    \frac{T_3 - \mu_{\bb}(T_3)}{(n-2)\mu_2 \sqrt{\binom{n}{2}}} \Rightarrow \mathcal{N}(0,v^2).
  \]
\end{cor}

Actually, the convergence can be quantified in the Wasserstein distance. Let us recall that for two probability measures $\mu, \nu$ on $\IR$ with finite first moment (i.\,e. $\int \abs{x} d\mu(x) < \infty, \int \abs{x} d\nu(x) < \infty$) the Wasserstein distance is defined as
\[
d_W(\mu,\nu) = \sup \left \lbrace \int f d\mu - \int f d\nu : f \in \mathrm{Lip}_1 \right \rbrace,
\]
where $\mathrm{Lip}_1$ denotes the set of all Lipschitz-continuous functions with Lipschitz constant at most $1$. For two random variables $X,Y$ we define $d_W(X,Y)$ as the Wasserstein distance between its distributions. Define
\begin{align*}
\tilde{T_3}(x) &\coloneqq (n-2)^{-1}\mu_2^{-1} \binom{n}{2}^{-1/2} \sum_{\{e,f,g\} \in \mathcal{T}_n} (x_{e} x_f x_g - \IE_{\mu_{\bb}}(x_{e} x_f x_g)) \\
\tilde{L}(x) &\coloneqq \binom{n}{2}^{-1/2} \sum_{e \in \mathcal{I}_n} (x_e - \IE_{\mu_{\bb}} x_e).
\end{align*}

\begin{proposition}	\label{proposition:WassersteinDistanceTriangles}
Let $\mu_{\bb} = \mu_{\bb}^{(n)}$ be an ERGM satisfying a $\partial$-LSI$(\sigma^2)$ and let $Z \sim \mathcal{N}(0,v^2)$ for some $v^2 > 0$. There exists a constant $C = C(\sigma^2)$ such that
\[
d_W(\tilde{T_3}, Z) \le d_W(\tilde{L}, Z) + C n^{-1/2}.
\]
\end{proposition}

Consequently, a rate of convergence in the Wasserstein distance for the number of edges in the ERGM immediately implies a rate of convergence for the number of triangles. 

\subsection{Central limit theorems for subgraph counts in the Erd{\"o}s--R{\'e}nyi model} \label{subsection:CLTER}
The second application of the concentration inequalities is a central limit theorem in the classical problem of subgraph counts in the Erd{\"o}s--R{\'e}nyi model. Note that the Erd{\"o}s--R{\'e}nyi model is a spin system $\mu_{n,p}$ on $\{0,1\}^{\mathcal{I}_n}$ with independent components and $\mu_{n,p}(x_e = 1) = p$.

Define the average degree (or average density) of a graph $d(G) \coloneqq \max_{H \subset G} \frac{\abs{E(H)}}{\abs{V(H)}}$, and the modified form $d'(G) \coloneqq \max_{H \subset G, \abs{E(H)} \ge 2} \frac{2\abs{E(H)}-1}{\abs{V(H)}-2}$. Denote by $T_G$ the number of subgraphs $G$ in the Erd{\"o}s--R{\'e}nyi random graph, i.\,e. graph homomorphisms from $G$ to the Erd{\"o}s--R{\'e}nyi random graph. A possible representation is
\begin{equation}    \label{eqn:T_G}
  T_G(X) = \frac{1}{\abs{\mathrm{Aut}(G)}} \sum_{f: V \to [n] \text{ injective }} \prod_{e \in E} X_{f(e)}
\end{equation}
with the definition $f(e) = \{ f(v_1), f(v_2) \}$ for $e = \{v_1, v_2\}$. Moreover, we write $(n)_k \coloneqq n(n-1) \cdots (n-k+1)$.

\begin{thm}      \label{theorem:ERCLTallgraphs}
  Let $G = (V,E)$ be any simple, connected graph. If $p = p(n)$ satisfies $p \le 1 - \varepsilon$ for some $\varepsilon > 0$ and $n p^{d'(G)}\log^{-\abs{E}}(1/p) \to \infty$, then
  \[
    \frac{T_G - \mu_{n,p}(T_G)}{2(\abs{\mathrm{Aut}(G)})^{-1} \abs{E} (n-2)_{\abs{V}-2} p^{\abs{E}-1} \sqrt{\binom{n}{2} p(1-p)}} \Rightarrow \mathcal{N}(0,1).
  \]
\end{thm}

We may replace the condition on $d'(G)$ in Theorem \ref{theorem:ERCLTallgraphs} by a condition on the maximal degree $d(G)$, though with an additional factor $6$, i.\,e. $np^{6d(G)} \log^{-\abs{E}}(1/p) \to \infty$ is a sufficient condition. This is a consequence of the inequality $d(G) \le d'(G) \le 6 d(G)$.

Similar calculations (and a proof of a central limit theorem for subgraph counts under non-optimal conditions) have been done in \cite{NW88}, interpreting subgraph counts as incomplete $U$-statistics and using the Hoeffding decomposition to prove the CLT. However, \cite[Theorem 3.1]{NW88} does not seem to be quite correct, since for triangles it requires a normalization of the subgraph count by $(n-2)\sqrt{p(1-p) \binom{n}{2}}$, which does not converge to a normal distribution in general. As can be seen from Theorem \ref{theorem:ERCLTallgraphs}, the correct normalization is $(n-2) p^2 \sqrt{p(1-p)\binom{n}{2}}$ (see also \cite[Equation (1.2)]{DE09}). In our approach, we additionally provide a quantification, i.e. we show that $T_3 - \IE_{\mu_{n,p}} T_3 - p^2(n-2)f_1 = \mathcal{O}_{\mu_{n,p}}(n^{3/2})$ with exponentially decaying tails. We do not believe it is possible to derive similar results using the $U$-statistics approach.

Using the method of moments, \cite{Ruc88} has shown that the convergence holds for any graph $G$ if and only if $np^{d(G)} \to \infty$ and $n^2(1-p) \to \infty$, which is optimal.  

\subsection{The general results} \label{subsection:CoMgeneralresults}
Theorems \ref{theorem:ERGMtriangleLpandTails} and \ref{theorem:ERCLTallgraphs} build upon general concentration properties in the presence of logarithmic Sobolev inequalities. To formulate our results, we introduce higher order differences
$\mathfrak{h}_{i_1 \ldots i_d}$ for any $d \in \mathbb{N}$ by setting $
\mathfrak{h}_{i_1 \ldots i_d}f = \mathfrak{h}_{i_1}(\mathfrak{h}_{i_2\ldots i_d}f).$ In particular, we obtain tensors of $d$-th order differences $\mathfrak{h}^{(d)}f$ with coordinates $\mathfrak{h}_{i_1\ldots i_d}f$. Regarding $\mathfrak{h}^{(d)}f$ as a vector indexed by $\mathcal{I}^d$, we may define $\abs{\mathfrak{h}^{(d)}f}$ as its Euclidean norm. We will write $\norm{f}_p$ for the $L^p(\mu)$ norm of $f$ and $\norm{\mathfrak{h}^{(d)} f}_p \coloneqq \norm{\abs{\mathfrak{h}^{(d)} f}}_p$.

\begin{thm}                      \label{theorem:com}
  Let $\mu$ be a spin system on $\mathcal{Y} = \mathcal{X}^{\mathcal{I}}$ satisfying a $\partial\textrm{-}\mathrm{LSI}(\sigma^2)$. For any $f: \mathcal{Y} \to \IR$ and $C \coloneqq \log(2)\sigma^2 (de)^2/2$ we have
  \begin{align}        \label{eqn:multilevelconcentration}
    \mu(\abs{f- \IE_\mu f} \ge t) \le 2 \exp\left(-\frac{1}{C} \min \left( \frac{t^{2/d}}{\norm{\mathfrak{h}^{(d)} f}_\infty^{2/d}}, \min_{k = 1,\ldots,d-1} \frac{t^{2/k}}{\norm{\mathfrak{h}^{(k)} f}^{2/k}_2} \right)\right).
  \end{align}
\end{thm}

We apply Theorem \ref{theorem:com} to functions which resemble multilinear polynomials, constructed in the following way. For any $d \in \IN$ we define the (generalized) diagonal of the index set $\mathcal{I}^d$ as
  $\Delta_d \coloneqq \{ (i_1,\ldots, i_d) \in \mathcal{I}^d : \abs{\{i_1,\ldots,i_d\}} < d \}.$
Let $f: \mathcal{X} \to \IR$ and $A$ a $d$-tensor with vanishing diagonal. We can associate to $f$ and any $J \subset \mathcal{I}$ the functions $f_J, \tilde{f}_J: \mathcal{Y} \to \IR$ defined via 
\[
f_J(y) = \prod_{i \in J} f(y_i) \quad \text{and} \quad \tilde{f}_i(y) = \prod_{i \in J} (f(y_i) - \int f(y_i) d\mu)
\]
and introduce the short-hand notations $\mu_J \coloneqq \IE_\mu f_J$ and $\tilde{\mu}_J \coloneqq \IE_\mu \tilde{f}_J$.
Given $(f,d,A)$, we construct polynomials as follows: for any finite set $\mathcal{J}$ let
\[
  \mathcal{P}(\mathcal{J}) = \left\lbrace S \subseteq 2^{\mathcal{J}} : S \text{ is a partition of } \mathcal{J} \right\rbrace.
\]
Finally, we set
\begin{align}    \label{eqn:fdApolynomial}
  f_{d,A} = \sum_{I \in \mathcal{I}^d} A_I \sum_{P \in \mathcal{P}(I)} g_P \coloneqq \sum_{I \in \mathcal{I}^d} A_I \sum_{P \in \mathcal{P}(I)} (-1)^{M(P)} \prod_{\substack{J \in P \\ \abs{J} = 1}} \tilde{f}_J  \prod_{\substack{J \in P \\ \abs{J} > 1}} \tilde{\mu}_J.
\end{align}

\begin{thm}                  \label{theorem:fdPolynomials}
  Let $\mu$ be a spin system on $\mathcal{Y} = \mathcal{X}^{\mathcal{I}}$ satisfying a $\partial\textrm{-}\mathrm{LSI}(\sigma^2)$, $d \in \IN, A$ a $d$-tensor with vanishing diagonal and $f: \mathcal{X} \to \IR$ with $\abs{f(x) - f(y)} \le c$ for all $x,y \in \mathcal{X}$. Then, $f_{d,A}$ as in \eqref{eqn:fdApolynomial} is a centered random variable, for all $p \ge 2$ we have
  \begin{align}    \label{eqn:fdALpestimate}
    \norm{f_{d,A}}_p \le \left( \sigma^2 c^2 \norm{A}_2^{2/d} p\right)^{d/2}
  \end{align}
  and consequently
  \begin{align}            \label{eqn:fdAestimates}
    \mu\left( \abs{f_{d,A}} \ge t \right) \le 2 \exp \left( - \Big( \frac{\log(2)t}{2e \sigma^d c^d \norm{A}_2} \Big)^{2/d} \right).
  \end{align}
\end{thm}

Note that although \eqref{eqn:fdALpestimate} seems like an estimate similar to integrability properties of eigenfunctions of the generator $L_\mu$ of the Glauber dynamics associated to $\mu$, we do not believe that it can be derived in a similar manner. The reason is that for a non-uniform measure $\mu$ the functions $f_{d,A}$ are not eigenfunctions of $L_\mu$, which can easily be seen in the case $d = 1$.

It is possible to refine the tail estimates further by considering different norms of $A$. As we shall not need it in the present work, we refer to \cite{AKPS18} and \cite{GSS18b}.

\section{Weak dependence of the models} \label{section:weakdependence}
We will finally introduce the concept of weakly dependent random variables. Let $\mu$ be a spin system on $\mathcal{Y} = \mathcal{X}^\mathcal{I}$. Define an \emph{interdependence matrix} $(J_{ij})_{i,j \in \mathcal{I}}$ as any matrix with $J_{ii} = 0$ and such that for any $x, y \in \mathcal{Y}$ with $\overline{x}_j = \overline{y}_j$ we have
\begin{equation}\label{eqn:DefiInterdependence}
  d_{\mathrm{TV}}(\mu(\cdot \mid \overline{x}_i), \mu(\cdot \mid \overline{y}_i)) \le J_{ij}.
\end{equation}
The matrix $J$ (or any norm thereof) may be interpreted as measuring the \emph{strength of the interactions between the spins} in the spin system $\mu$. (In particular, note that $J = 0$ is an interdependence matrix for product measures $\mu$.)
Moreover, we need to control the minimal probabilities of the marginal distributions of the spin system $\mu$. To this end, define for any subset $S \subsetneq \mathcal{I}$ and any $i \notin S$
\begin{equation}\label{eqn:DefiBetaTilde}
  \tilde{\beta}_{i,S}(\mu) \coloneqq \inf_{\substack{x_S \in \mathcal{X}^S \\ \mu_S(x_S) > 0}} \inf_{\substack{y_{S^c} \in \mathcal{X}^{S^c} \\ \mu(y_{S^c}, x_{S}) > 0} }  \mu ( (y_{S^c})_i \mid x_S).
\end{equation}
If $S = \emptyset$, this reads $\tilde{\beta}_{i,\emptyset}(\mu) = \inf_{y \in \mathcal{Y} : \mu(y) > 0} \mu(y_i).$ The interpretation of $\tilde{\beta}_{i,S}(\mu)$ is straightforward: For any admissible partial configuration $x_S \in \mathcal{X}^S$ all possible marginals are supported on points with probability at least $\tilde{\beta}_{i,S}(\mu)$. Now let
\begin{equation*} 
  \tilde{\beta}(\mu) \coloneqq \inf_{S \subsetneq \mathcal{I}} \inf_{i \notin S} \tilde{\beta}_{i,S}(\mu)
\end{equation*}
be the infimum of all $\tilde{\beta}_{i,S}(\mu)$. Note that if there are no hard constraints, i.e. $\mu$ has full support, we have $\tilde{\beta}(\mu) = I(\mu) = \min_{i \in \mathcal{I}} \min_{y \in \mathcal{Y}} \mu( y_i \mid \overline{y}_i).$

We say that $\mu$ is \emph{$(\alpha_1, \alpha_2)$-weakly dependent}, if for some interdependence matrix $J$ 
\[
	\tilde{\beta}(\mu) \ge \alpha_1 \qquad \text{and} \qquad \norm{J}_{2 \to 2} \le 1-\alpha_2.
\]

\begin{thm}                          \label{theorem:LSIforSpinSystems}
  Let $\mu$ be a $(\alpha_1, \alpha_2)$-weakly dependent spin system.

  \begin{enumerate}
	\item For any function $f: \mathcal{Y} \to \IR_+$ vanishing outside of $\supp \mu$ we have
  	\begin{align}	\label{thm:eqn:ConclusionRelativeEntropy2}
  		\Ent_{\mu}(f) \le \frac{1}{\alpha_1\alpha_2^2} \sum_{i \in \mathcal{I}} \int \Ent_{\mu(\cdot \mid \overline{x_i})}(f(\overline{x_i}, \cdot)) d\mu(x).
  	\end{align}
  	\item $\mu$ satisfies $\mathrm{mLSI}(\rho_0)$ for $\rho_0 \coloneqq 2\abs{\mathcal{I}} \alpha_1^{-1} \alpha_2^{-2}$.
  	\item $\mu$ satisfies $\mathrm{LSI}(\sigma^2)$ for $\sigma^2 \coloneqq \log(\alpha_1^{-1})(\log(2) \alpha_1 \alpha_2^2)^{-1}$. 
  \end{enumerate}
\end{thm}

\begin{proof}[Proof of Theorem \ref{theorem:LSIforSpinSystems}]
  $(1)$: The entropy tensorization property has been established in \cite{Ma15} (see also \cite[Theorem 4.2]{GSS18} with a better constant). The condition on the full support of $\mu$ can be weakened by adapting the definition of $\beta$ in \cite[Theorem 4.2]{GSS18}.
  
  $(2)$: Let us define $\Omega_0 \coloneqq \supp \mu$, where $\supp$ is the support of $\mu$, i.e. $\supp(\mu) \coloneqq \{ y \in \mathcal{Y}: \mu(y) > 0\}$. The first part yields for any $f: \mathcal{Y} \to \IR$ vanishing outside of $\Omega_0$
  	\begin{equation}	\label{eqn:modLSI}
  		\Ent_{\mu}(f) \le \frac{1}{\alpha_1 \alpha_2^2} \sum_{i \in \mathcal{I}} \int \Ent_{\mu(\cdot \mid \overline{x_i})}(f(\overline{x_i},\cdot)) d\mu(x).
  	\end{equation}
  	This is equivalent to the fact that on the probability space $(\Omega_0, \mu)$, any function $f: \Omega_0 \to \IR_+$ satisfies the same inequality, which we shall work with from now on. For any probability measure $(\Omega, \mathcal{F},\nu)$ and any function $f$ such that $f, e^f \in L^2(\nu)$, we have by Jensen's inequality and the symmetry in the covariance
  	\begin{equation}	\label{eqn:proofmodLSIeqn2}
  		\Ent_{\nu}(e^f) \le \Cov_{\nu}(f,e^f) = \int \left( \int (f(y) - f(x)) d\nu(x) \right) e^{f(y)} d\nu(y).
  	\end{equation}
  	Apply the inequality \eqref{eqn:proofmodLSIeqn2} to the integral on the right hand side of equation \eqref{eqn:modLSI} to get
  	\begin{equation}	\label{eqn:modLSIWithCondProb}
  		\Ent_{\mu}(e^f) \le \frac{1}{\alpha_1 \alpha_2^2} \sum_{i \in \mathcal{I}} \int \left( \int (f(x) - f(\overline{x_i},y)) d\mu(y \mid \overline{x_i}) \right) e^{f(x)} d\mu(x).
  	\end{equation}
  	Finally, observe that for the transition matrix $P$ and the generator $-L = I - P$ of the Glauber dynamics (on $\Omega_0$) we have
  	\begin{align*}
  		\mathcal{E}(e^f,f) &= \IE_{\mu}(e^f(-Lf)) = \int \sum_{y \in \Omega_0} (f(x) - f(y)) P(x,y) e^{f(x)} d\mu(x) \\
  		&= \frac{1}{\abs{\mathcal{I}}} \sum_{i \in \mathcal{I}} \iint (f(x) - f(\overline{x_i}, y)) d\mu(y \mid \overline{x_i}) e^{f(x)} d\mu(x),
  	\end{align*}
  	so that \eqref{eqn:modLSIWithCondProb} may be rewritten as $\Ent_{\mu}(e^f) \le \frac{\abs{\mathcal{I}} }{\alpha_1 \alpha_2^2} \mathcal{E}(e^f,f)$, which yields the claim.
  
  $(3)$: For any $i \in \mathcal{I}$, any $y \in \mathcal{Y}$ with $\mu(y) > 0$ the measure $\mu(\cdot \mid \overline{y}_i)$ is a measure on $\mathcal{X}$ with $\min_{x \in \mathcal{X}} \mu(x \mid \overline{y}_i) \ge \alpha_1$, and so \cite[Remark 6.6]{BT06} yields
  \[
    \Ent_{\mu(\cdot \mid \overline{y}_i)}(g^2) \le \frac{-2 \log(\alpha_1)}{\log(2)} \Var_{\mu(\cdot \mid \overline{y}_i)}(g),
  \]
  which plugged into equation \eqref{thm:eqn:ConclusionRelativeEntropy2} leads to
  \begin{align*}
    \Ent_{\mu}(f^2) \le 2 \frac{-\log(\alpha_1)}{\log(2) \alpha_1 \alpha_2^2} \sum_{i \in \mathcal{I}} \int \Var_{\mu(\cdot \mid \overline{y}_i)}(f(\overline{y}_i,\cdot))  d\mu(y) = 2 \sigma^2 \int \abs{\partial f(y)}^2 d\mu(y).
  \end{align*}
\end{proof}

Next, we prove that all the models defined in Section \ref{subsection:themodels} satisfy the conditions of Theorem \ref{theorem:LSIforSpinSystems}.

\begin{proposition}                  \label{proposition:LSIinERGM}
  If $\bb$ is such that $\frac{1}{2} \Phi'_{\abs{\beta}}(1)< 1$, then $\mu_{\bb}$ satisfies the conditions of Theorem \ref{theorem:LSIforSpinSystems}. 
\end{proposition}

\begin{proof}[Proof of Proposition \ref{proposition:LSIinERGM}]
  It is convenient to introduce some notation for the exponential random graph model. For any graph $x \in \mathcal{G}_n$ and any edge $e = (i,j) \in \mathcal{I}_n$ let $x_{e+}$ (resp. $x_{e-}$) be the graph with edge set $E(x_{e+}) = E(x) \cup e$ ($E(x_{e-}) = E(x) \backslash e$ respectively). For any function $f: \mathcal{G}_n \to \IR$ we define the discrete derivative in the $e$-th direction as $\partial_e f(x) = f(x_{e+}) - f(x_{e-}).$  More generally, given edges $e_1, \ldots, e_k$ we define $\partial_{e_1 \cdots e_k}$ recursively, i.e. $
    \partial_{e_1 \cdots e_k} f = \partial_{e_1}\left( \partial_{e_2 \cdots e_k} f \right).$
  It is easy to see that the definition does not depend on the order of the edges and $\partial_{ee} f = 0$. The partial derivatives of the Hamiltonian are given by
  \begin{align*}
    \partial_e H(x) & = 2\beta_1 + n^2 \sum_{i = 2}^s \frac{\beta_i}{n^{\abs{V_i}}} (N_{G_i}(x_{e+}) - N_{G_i}(x_{e-})).
  \end{align*}
  Now we use the fact if $G_i$ injects into $x_{e-}$, then it also injects into $x_{e+}$, and hence the sum is only nonzero if the edge $e$ is essential for the injection, and write $N_{G_i}(x,e)$ to denote the number of injections of $G_i$ into $x$ which use the edge $e \in E(x)$, so that $\partial_e H(x) = 2\beta_1 + n^2 \sum_{i = 2}^s \frac{\beta_i}{n^{\abs{V_i}}} N_{G_i}(x,e)$. Especially this gives
  $\abs*{\partial_e H(x)} = O(1)$. \par
  We want to apply Theorem \ref{theorem:LSIforSpinSystems}. The spin system is given by $\mathcal{Y}_n \coloneqq \{ 0,1 \}^{\mathcal{I}_n}$, and $\mu_n$ is the push-forward of the measure associated to the exponential random graph model $\ERGM(\bb, G_1, \ldots, G_s)$ on $\mathcal{G}_n$. The condition on the conditional distributions is easy to check, since for any $e \in \mathcal{I}_n$ and any $y \in \mathcal{Y}_n$
  \[
    \mu_n(y_e \mid \overline{y}_e) = \frac{1}{2}(1 + \tanh(\partial_e H(y)/2))
  \]
  and $\partial_e H(y) = O(1)$, where the constant depends on $({\abs{\bb}}, G_1, \ldots, G_s)$ only. Hence it remains to prove the second condition. To this end, let again $x = x_{f+}, y = x_{f-}$ be two graphs which differ in one edge $f$ only, and observe that for each other edge $e$
  \begin{align*}
    d_{\mathrm{TV}}(\mu_n(\cdot \mid \overline{x}_e), \mu_n(\cdot \mid \overline{y}_e)) 
    &= \frac{1}{2} \abs{\tanh(\partial_e H(x_{f+})/2) - \tanh(\partial_e H(x_{f-})/2)} \\
    &\le \frac{1}{4} \abs{\partial_{fe} H(x)} \le \frac{n^2}{4} \sum_{i = 2}^s \abs{\beta_i} \frac{N_{G_i}(x,f,e)}{n^{\abs{V_i}}} \\
    &\le \frac{n^2}{4} \sum_{i = 2}^s \abs{\beta_i} \frac{N_{G_i}(K_n,f,e)}{n^{\abs{V_i}}},
  \end{align*}
  i.e. $J_{fe} \le \frac{n^2}{4} \sum_{i = 2}^s \abs{\beta_i} \frac{N_{G_i}(K_n,f,e)}{n^{\abs{V_i}}}$. Thus after summation in $f \in \mathcal{I}_n$ we obtain by \cite[Lemma 9(c)]{BBS11}
  \begin{align*}
    \sum_{f \neq e} J_{fe} \le \frac{n^2}{4} \sum_{i = 2}^s \abs{\beta_i} \sum_{f \neq e} \frac{N_{G_i}(K_n,f,e)}{n^{\abs{V_i}}} = \frac{1}{2} \sum_{i = 2}^s \abs{\beta_i} \abs{E_i}(\abs{E_i}-1) = \frac{1}{2} \Phi'_{\abs{\beta}}(1).
  \end{align*}
  Since the right-hand side is independent of $e \in \mathcal{I}_n$, this yields $\norm{J}_{1 \to 1} \le \frac{1}{2} \Phi_{\abs{\beta}}'(1) < 1$. Moreover, $J$ is a symmetric matrix, so that we have $\norm{J}_{2 \to 2} \le \norm{J}_{1 \to 1}$.
\end{proof}

\begin{proposition}		\label{proposition:LSIinVertexWeighted}
Let $\mu_{\beta}$ be the vertex-weighted exponential random graph model and assume that $\sup_{\lambda \in (0,1)} \abs{\phi_\beta'(\lambda)} < 1$. $\mu_{\beta}$ satisfies the conditions of Theorem \ref{theorem:LSIforSpinSystems}.
\end{proposition}

\begin{proof}[Proof of Proposition \ref{proposition:LSIinVertexWeighted}]
Since $x_i \in \{0,1\}$ implies $x_i^k = x_i$ for all $k \in \IN$, we can rewrite the Hamiltonian using the order parameter $S \coloneqq \sum_{i = 1}^n x_i$ as
\[
\mu(x) = Z^{-1} \exp \left( \frac{\beta_1}{n} S(S-1) + \frac{\beta_2}{n^2} S(S-1)(S-2) + \log \frac{p}{1-p} S \right).
\]
Hence for $\mathcal{X} \coloneqq \{0,1\}$ and $\mathcal{I}_n \coloneqq \{1,\ldots,n\}$ we are in the situation of Theorem \ref{theorem:AbstractResult}, and it remains to check conditions \eqref{eqn:LowerBoundConditionalProbability} and \eqref{eqn:UpperBoundInterdependenceMatrix}. Observe that we have (with the same notations as in the exponential random graph models)
\[
\mu(1 \mid \overline{x_e}) = \frac{\exp(\partial_e H_n(\overline{x}_e,1))}{1 + \exp(\partial_e H_n(\overline{x}_e,1))} = \frac{1}{2}\left( 1 + \tanh\left( \partial_e H_n(x)/2 \right) \right),
\]
where in this case $\abs{\partial_e H_n(x)} = \abs{\frac{2\beta_1}{n} \sum_{i \neq e} x_i + \frac{3\beta_2}{n^2} \sum_{i \neq j, i,j \neq e} x_ix_j + \log(p/(1-p))}$ is bounded by a constant depending on $\beta$, so that a lower bound on the conditional probabilities holds. The inequality \eqref{eqn:UpperBoundInterdependenceMatrix} is already implicitly proven in the proof of \cite[Lemma 6]{DEY17}, which we modify. Fix a site $e \in \mathcal{I}_n$ and two configurations $x,y$ differing solely at $f \in \mathcal{I}_n$, i.e. $x_f = 1, y_f = 0$, and let $S \coloneqq \sum_{i = 1}^n y_i$. We have
\[
	d_{\mathrm{TV}}(\mu(\cdot \mid \overline{x}_e), \mu(\cdot \mid \overline{y}_e)) = \frac{1}{2}\abs{\tanh(\partial_e H_n(\overline{x}_e, 1)) - \tanh(\partial_e H_n(\overline{y}_e,1))}
\]
and since $H_n$ (and as a consequence $\partial_e H_n$) only depends on the sum $S$ of a vector, by defining $h_n(\lambda) \coloneqq \beta_1 \lambda + \beta_2 \lambda^2 - \frac{\beta_2}{n} \lambda + \log(p/(1-p))$ we can estimate for some $\xi \in (0,1)$
\begin{align}	\label{eqn:JfeEstimate}
	J_{fe} \le \abs*{\frac{\exp(h_n((S+1)/n))}{1+\exp(h_n((S+1)/n))} - \frac{\exp(h_n(S/n))}{1+\exp(h_n(S/n))}} = \frac{1}{n} \abs*{\left(\frac{\exp \circ h_n}{1+\exp \circ h_n}\right)'(\xi)}.
\end{align}
Lastly, if we define $h(\lambda) = \beta_1 \lambda + \beta_2 \lambda^2 + \log(p/(1-p))$, using the Lipschitz property of the function $\exp(x)/(1+\exp(x))$ it can be shown that 
\[
	\abs*{ \frac{\exp \circ h_n}{1+\exp \circ h_n} - \frac{\exp \circ h}{1 + \exp \circ h} } = O(n^{-1})
\]
and $h_n$ can be replaced by $h$ in \eqref{eqn:JfeEstimate} with an error of $O(n^{-2})$. By summing up over $f \neq e$, we obtain for $n$ large enough and all parameters such that 
\begin{align}	\label{eqn:conditionOnDerivative}
\sup_{\lambda \in (0,1)} \abs*{ \frac{\exp \circ h}{1 + \exp \circ h}' } < 1
\end{align}
that \eqref{eqn:UpperBoundInterdependenceMatrix} holds.
\end{proof}

\begin{rem*}
Note that the condition \eqref{eqn:conditionOnDerivative} can be written in terms of the functions defined for exponential random graphs. More specifically, we have for any $x \in \IR$
\[
	\frac{\exp(h(x))}{1 + \exp(h(x))} = \phi_{\tilde{\beta}_1,\tilde{\beta}_2,\tilde{\beta}_3}(x)
\]
for the ERGM $\mu_\beta$ given by the three parameters $\tilde{\beta}_1 = \frac{\log(p/(1-p))}{2}$, $\tilde{\beta}_2 = \frac{\beta_1}{4}, \tilde{\beta}_3 = \frac{\beta_2}{6}$ and the three graphs $G_1$ an edge, $G_2$ a $2$-star and $G_3$ a triangle.
\end{rem*}

\begin{proposition}      \label{proposition:LSIinrandomcoloring}
  Let $G_n = (V_n, E_n)$ be a sequence of graphs with uniformly bounded maximum degree $\Delta$ and $k \ge 2\Delta +1$. Then the conditions of Theorem \ref{theorem:LSIforSpinSystems} hold for the random coloring model $\mu_{G_n}$ with $\alpha_1, \alpha_2$ independent of $n$.
\end{proposition}

\begin{proof}[Proof of Proposition \ref{proposition:LSIinrandomcoloring}]
  This is again an application of Theorem \ref{theorem:LSIforSpinSystems}. Let us first show that
    $J_{v,w} \coloneqq \frac{1}{\Delta+1} \eins_{v \sim w}$
  can be used as an interdependence matrix. To see this, let $c^1, c^2 \in \Omega_0$ be two colorings that differ only in one vertex $v_1$, and $v_2$ be another vertex. In the case $v_1 \sim v_2$ (in $G_n$) the measures $\mu_{v_2}( \cdot \mid \overline{c^i}_{v_2} )$ are uniform on $C \backslash \{ c^i_{v_k} : v_k \sim v_1 \}$ for $i = 1,2$, and hence
  \begin{align*}
     & d_{\mathrm{TV}}(\mu_{v_2}( \cdot \mid \overline{c^1}_{v_2}), \mu_{v_2}( \cdot \mid \overline{c^2}_{v_2})) \\ &= \frac{1}{2} \left( \frac{1}{k-\abs{\{ \overline{c^1}_{v_2} : v_2 \sim v_1 \}}} + \frac{1}{k-\abs{\{ \overline{c^2}_{v_2} : v_2 \sim v_1 \}}} \right) \le \frac{1}{k-\Delta} \le \frac{1}{\Delta+1}.
  \end{align*}
  On the other hand, if $v_2 \not\sim v_1$, then $\mu_{v_2}(\cdot \mid \overline{c^i}_{v_2})$ are equal and thus $J_{v_1, v_2} = 0$. Since $J$ is a symmetric matrix, we obtain
  \[
    \norm{J}_{2 \to 2} \le \norm{J}_{1 \to 1} \le \max_{v_j \in V_n} \sum_{v_i \in V_n} J_{v_i,v_j} \le \frac{\Delta}{\Delta+1} < 1.
  \]
  Moreover, we have to show that $\tilde{\beta}(\mu_n) \ge \alpha_1$ uniformly in $n \in \IN$. Let $S \subsetneq V_n, S \neq \emptyset$ be arbitrary, $v_1 \notin S$ and $c^S \in C^S$ be a proper coloring of $G\mid_S = (S, E_n \cap S \times S)$ and $c^{v_1} \in C \backslash \{ c_{v_2} : v_2 \in S, v_2 \sim v_1 \}$. Using the definition $\Omega_0(G)$ for the set of all proper colorings of an arbitrary graph $G$ (with a fixed number of colors, here $k$), we have
  \begin{align}  \label{eqn:probColoring}
    \mu_{\overline{S}}(c^{v_1} \mid c^S) = \frac{\mu(c^{v_1}, c^S)}{\mu(c^S)} = \frac{\abs{\Omega_0(G \mid_S)}}{\abs{\Omega_0(G \mid_{S \cup v_1})}}.
  \end{align}
  It is clear that $\abs{\Omega_0(G\mid_S)} = \frac{1}{\abs{C}} \abs{\Omega_0(\tilde{G_S})}$, where $\tilde{G_S}$ is obtained by adding an isolated vertex $v_1$ to $S$. Hence we fix the vertex set $S \cup v_1$ and rewrite equation \eqref{eqn:probColoring} as follows. Let $N(v_1, S) = \{ v_2 \in S : v_2 \sim v_1 \} = \{e_1, \ldots, e_l \}$ be the neighbors of $v_1$ in $S$ and for any $e_1, \ldots, e_k \in N(v_1, S)$ let $G_{e_1,\ldots, e_k}$ be the graph with edge set $(E_n \cap S \times S) \cup \{e_1, \ldots, e_k\}$, so that
  \begin{align*}
    \mu_{\overline{S}}(c^{v_1} \mid c^S) = \frac{1}{\abs{C}} \prod_{k = 1}^{l} \frac{\abs{\Omega_0(G_{e_1, \ldots, e_{k-1})}}} {\abs{\Omega_0(G_{e_1, \ldots, e_{k})}}}.
  \end{align*}
  By \cite[equation (2)]{Jer95}, it follows that each of the ratios is bounded from below by a constant depending on $\Delta$, thus resulting in
   $ \mu_{\overline{S}}(c^{v_1} \mid c^S) \ge \abs{C}^{-1} c(\Delta),$
  with a possible choice $c(\Delta) = \left( \frac{\Delta+1}{\Delta+2} \right)^{\Delta}$. The case $S = \emptyset$ is easier, as $\mu_i(c^i) = \frac{1}{\abs{C}}$ by the invariance of the random coloring model induced by a relabeling of the colors $C$.
\end{proof}

\begin{proposition}      \label{proposition:LSIinhardcore}
  Let $G = (V,E)$ be any graph with maximum degree $\Delta$. The conditions of Theorem \ref{theorem:LSIforSpinSystems} hold for the hard core model $\mu_{\lambda}$ with fugacity $\lambda$, and $\alpha_1, \alpha_2$ depend on $\Delta$ only.
\end{proposition}

\begin{proof}[Proof of Proposition \ref{proposition:LSIinhardcore}]
  Since we are going to require hard-core models corresponding to various graphs, we will write $\mu_G$ to emphasize the graph under consideration. The fugacity $\lambda$ will not change. Let us show that $J_{v_1,v_2} = \frac{\lambda}{1+\lambda} \eins_{v_1 \sim v_2}$ can be used as an interdependence matrix. Let $v_1 \in V$ be a site, $\sigma^1, \sigma^2 \in \mathcal{Y}$ be two admissible configurations differing only at site $v_1$ (without loss of generality $\sigma^1_{v_1} = 1, \sigma^2_{v_1} = 0$), and $v_2 \in V$ be another site. If $v_2 \sim v_1$, then $\mu_G(1 \mid \overline{\sigma^1_{v_1}}) = 0$, whereas $\mu_G(1 \mid \overline{\sigma^1_{v_1}}) = \frac{\lambda}{1+\lambda}$. If $v_2 \not\sim v_1$ we have $\mu_G(\cdot \mid \overline{\sigma^1_{v_1}}) = \mu_G(\cdot \mid \overline{\sigma^2_{v_1}})$. Hence by the symmetry of $J$
  \[
    \norm{J}_{2 \to 2} \le \norm{J}_{1 \to 1} \le \Delta \frac{\lambda}{1+\lambda} < 1
  \]
  which is a consequence of $\lambda < \frac{1}{\Delta - 1}$. \par
  To see that there is a lower bound on the conditional probabilities, let us first consider the case $S = \emptyset$. Let $v \in V$ be arbitrary, write $N(v)$ for the neighborhood of $v$ and $A$ for the complement of $v \cup N(v)$, and observe that
  \begin{align*}
    \mu_G(\sigma_v = 1) & = \mu(\sigma_v = 1, \sigma_{N(v)} = 0) = Z^{-1} \sum_{\tilde{\sigma}_{A}} \mu(\sigma_v = 1, \sigma_{N(v)} = 0, \sigma_{A} = \tilde{\sigma}_{A}) \\
                        & = \lambda Z^{-1} \sum_{\tilde{\sigma}_A} \lambda^{\abs{\tilde{\sigma}_A}} \eqqcolon \lambda Z^{-1} Z_A,
  \end{align*}
  where the summation is over all admissible configurations. Note that due to $\sigma_{N(v)} = 0$ these are actually all admissible configurations of the graph $G\mid_A = (A, E \cap A \times A)$. The normalizing constant can be bounded from above and below by
  \begin{align*}
    Z = \sum_{\tilde{\sigma}_A \text{ adm.}} \lambda^{\abs{\tilde{\sigma}_A}} \sum_{\tilde{\sigma}_{A^c}} \lambda^{\abs{\tilde{\sigma}_{A^c}}} \eins_{(\tilde{\sigma}_A, \tilde{\sigma}_{A^c}) \text{ adm.}} \le 2^{\Delta+1} Z_A
  \end{align*}
  and $ Z \ge (\lambda + 1) Z_A,$ which follows by only considering the configurations $\sigma_v = 1, \sigma_{N(v)} = 0$ and $\sigma_{v \cup N(v)} = 0$. As a consequence, we have
  \begin{align}                \label{eqn:UpperLowerBound}
    \frac{\lambda}{2^{\Delta+1}} \le \mu(\sigma_i = 1) \le \frac{\lambda}{\lambda+1}.
  \end{align}
  The case $S \neq \emptyset$ follows by a reduction argument. Let $\tilde{\sigma}_S$ be an admissible configuration of $G\mid_S$ and let $T \coloneqq \{ w \in S : \sigma_w = 0 \} \subset S$ be the free sites in $S$. By explicitly calculating the conditional probability one can see that for any configuration $\sigma_{S^c}$ and any $v \notin S$ we have
   $ \mu_G(\sigma_v = 1 \mid \sigma_S = \tilde{\sigma}_S) = \mu_{G\mid_{V \backslash T}}(\sigma_v = 1 \mid \sigma_{S \backslash T} = (1,\ldots,1)).$
  The graph $G\mid_{V\backslash T}$ can be decomposed into three parts: $(T, N(T), R)$, where $N(T) = \cup_{v \in T} N(v)$ and
  \[
    \mu_G(\sigma_v = 1 \mid \sigma_S = \tilde{\sigma}_S) = \mu_{G \mid_R}(\sigma_v = 1),
  \]
  which has an upper and lower bound by inequality \eqref{eqn:UpperLowerBound} and so we ultimately get $\tilde{\beta}(\mu_G) \ge c(\Delta, \lambda)$.
\end{proof}

\section{Proofs of the mixing time results} \label{section:proofMixingTimes}
With the results of the last section, the proof of the mixing time results is straightforward.

\begin{proof}[Proof of Theorem \ref{theorem:GlauberDynamicsRapidMixing}]
	Propositions \ref{proposition:LSIinERGM}, \ref{proposition:LSIinVertexWeighted}, \ref{proposition:LSIinrandomcoloring} and \ref{proposition:LSIinhardcore} show that all the models satisfy the conditions of Theorem \ref{theorem:LSIforSpinSystems}. Thus, a $\mathrm{mLSI}(\rho_0)$ with $\rho_0 \le C \abs{\mathcal{I}_n}$ holds, where the constant depends on the underlying parameters, but not on $n$. Hence the assertion follows immediately from Theorem \ref{theorem:AbstractResult}.
\end{proof}

\begin{proof}[Proof of Theorem \ref{theorem:AbstractResult}]
	The $\mathrm{mLSI}(\rho_0)$ property follows immediately from Theorem \ref{theorem:LSIforSpinSystems}. Equation \eqref{eqn:ExponentialDecayRelativeEntropy} is a consequence of \cite[Theorem 2.4]{BT06}, noting that our $\mathrm{mLSI}$ constant $\rho_0$ corresponds to $1/\rho_0$ in \cite{BT06}.
		
	Now let $(\mu_n)_n$ be a sequence of spin systems with sites $(\mathcal{I}_n)_n$, spins $\mathcal{X}$, and define $\mathcal{Y}_n = \supp(\mu_n) \subset \mathcal{X}^{\mathcal{I}_n}$. To prove rapid mixing, note that 
	\[
		\frac{2}{\rho_0} = \inf \left\lbrace \frac{\mathcal{E}(e^f,f)}{\Ent_{\mu_n}(e^f)} : f \neq \mathrm{const} \right\rbrace \ge \frac{\alpha_1 \alpha_2^2}{2 \abs{\mathcal{I}_n}}.
	\]
	If we denote $\mu_n^* = \min_{y \in \mathcal{Y}_n} \mu_n(y)$, \cite[Corollary 2.8]{BT06} yields
	\[
		d_{\mathrm{TV}}(P^t(y,\cdot), \mu_n)^2 \le 2 \log\Big( \frac{1}{\mu_n^*} \Big) \exp(-2 \rho_0^{-1} t) \le 2 \log\Big( \frac{1}{\mu_n^*} \Big) \exp\Big(- \frac{\alpha_1 \alpha_2^2}{2 \abs{\mathcal{I}_n}} t \Big).
	\]
	Hence for $t = \frac{2\abs{\mathcal{I}_n}}{\alpha_1 \alpha_2^2} \cdot \left( \log 2 + 2 + \log \log ( 1/\mu_n^* ) \right)$ we have $\max_{y \in \mathcal{Y}_n} d_{\mathrm{TV}}(P^t(y,\cdot), \mu_n)^2 \le e^{-2}$, i.e. $t_{\mathrm{mix}}(n) \le \frac{2\abs{\mathcal{I}_n}}{\alpha_1 \alpha_2^2} \cdot \left( \log 2 + 2 + \log \log ( 1/\mu_n^*) \right)$. 
	
	The last step is to show $\log \log 1/\mu_n^* = O(\log \abs{\mathcal{I}_n})$. This follows using the definition of $\alpha_1$, since by conditioning and iterating we obtain for any $y \in \mathcal{Y}_n$
	\[
		\frac{1}{\mu_n(y)} = \frac{1}{\mu_n(y_{i} \mid \overline{y_i})} \mu_n(\overline{y_i})^{-1} \le \alpha_1^{-1} \mu_n(\overline{y_i})^{-1} \le \ldots \le \alpha_1^{-\abs{\mathcal{I}_n}}.
	\]
	Hence, $t_{\mathrm{mix}} = O(\abs{\mathcal{I}_n} \log \abs{\mathcal{I}_n})$.
\end{proof}

\begin{proof}[Proof of Corollary \ref{corollary:RapidMixingGibbsMeasure}]
If there are no hard constraints, i.e. $\mu$ has full support, then $\tilde{\beta}(\mu)$ can be simplified to
\[
	\tilde{\beta}(\mu) = I(\mu) \coloneqq \min_{i \in \mathcal{I}} \min_{y \in \mathcal{Y}} \mu_n( y_i \mid \overline{y_i}).
\]
This can be shown by conditioning for any $S \subset \mathcal{I}$ and any $x_S \in \mathcal{X}^S$
\begin{align*}
\mu(y_i \mid x_S) = \mu(x_S)^{-1} \sum_{z \in \mathcal{X}^{\mathcal{I} \backslash (S \cup i)}} \mu(y_i \mid x_S, z) \mu(x_S, z) \ge I(\mu),
\end{align*}
and the reverse inequality follows by taking $S = \mathcal{I} \backslash \{ j \}$.
\end{proof}

\section{Proofs of the concentration of measure results}    \label{section:ProofsCoM}

\begin{proof}[Proof of Theorem \ref{theorem:LSIforAllModels}]
The proof is an application of Theorem \ref{theorem:LSIforSpinSystems}, see the proof of Theorem \ref{theorem:GlauberDynamicsRapidMixing}.
\end{proof}

We will use Theorem \ref{theorem:fdPolynomials} to prove several results for finite spin systems. 
The most important cases of Theorem \ref{theorem:fdPolynomials} will be $d = 1,2,3$. It is easy to check that
\begin{align*}
  f_{1,A} & \coloneqq \sum_{i \in \mathcal{I}} A_i \tilde{f}_i, \displaybreak[2] \qquad
  f_{2,A} \coloneqq \sum_{i,j \in \mathcal{I}} A_{ij} (\tilde{f}_i \tilde{f}_j - \tilde{\mu}_{ij}), \displaybreak[2]\\
  f_{3,A} & \coloneqq \sum_{i,j,k \in \mathcal{I}} A_{ijk} \left( \tilde{f}_i \tilde{f}_j \tilde{f}_k - \tilde{\mu}_{ijk}- \tilde{f}_i \tilde{\mu}_{jk} - \tilde{f}_j \tilde{\mu}_{ik} - \tilde{f}_k \tilde{\mu}_{ij} \right).
\end{align*}

\subsection{Exponential random graph model}
We apply the general result on the concentration of the polynomials $f_{d,A}$ from \eqref{eqn:fdApolynomial} with the spin function $f(x) = x$. Before we prove Theorem \ref{theorem:ERGMtriangleLpandTails} this way (corresponding to $d=3$), let us give a simple example which already demonstrates some of the arguments we will use.

\begin{exa*}\label{exa:ergmd12}
  Let $\mu_{\bb}$ be an ERGM satisfying a $\partial$-$\textrm{LSI}(\sigma^2)$, and let $T_1(x) \coloneqq \sum_{e \in \mathcal{I}_n} x_e$ be the number of edges. Moreover, for any two disjoint subsets $S_1, S_2 \subset \{1,\ldots,n\}$, write
  $
    C(S_1,S_2) \coloneqq \{ e = (i,j) \in \mathcal{I}_n : \{i,j \} \cap S_1 \neq \emptyset, \{i,j\} \cap S_2 \neq \emptyset \}
  $
  and let $T_{S_1,S_2} \coloneqq \sum_{e \in \mathcal{I}_n} \eins_{C(S_1,S_2)}(e) x_e$ be the number of edges between $S_1$ and $S_2$.
  Then,
  \begin{align}
    \mu_{\bb}(\abs{T_1 - \IE_{\mu_{\bb}(T_1)}} \ge t)                       & \le  2 \exp \left( - \frac{\log(2) t^2}{e^2 \sigma^2 n(n-1)} \right),\label{eqn:countingedges}                  \\
    \mu_{\bb}(\abs{T_{S_1,S_2} - \IE_{\mu_{\bb}}(T_{S_1,S_2})} \ge t) & \le 2 \exp \left( - \frac{\log(2) t^2}{2 e^2 \sigma^2 \abs{S_1} \abs{S_2}} \right)\label{eqn:countingedgesbds}.
  \end{align}
  In particular, setting $\eta \coloneqq \IE_{\mu_{\bb}}(x_e)$ for $e \in \mathcal{I}_n$ arbitrary, we obtain a strong law of large numbers, i.\,e. $T_1/\abs{\mathcal{I}_n} \to \eta$ a.s.
\end{exa*}

\begin{proof}
  Noting that $T_1 - \IE_{\mu_{\bb}}(T_1) = f_{1,A}$ for $A = (1,\ldots,1)$, \eqref{eqn:countingedges} readily follows from Theorem \ref{theorem:fdPolynomials}. Equation \eqref{eqn:countingedgesbds} follows by similar arguments. 
  To prove the strong law of large numbers, first note that by the intrinsic symmetry (i.\,e. a relabeling of the vertices $\{1,\ldots,n\}$ and a respective relabeling of the edges will result in the same probability law), it is easy to see that $\mu_{\bb}(x_e)$ does not depend on $e \in \mathcal{I}_n$. Thus, $\eta$ is well-defined and $\IE_{\mu_{\bb}}(T_1) = \abs{\mathcal{I}_n} \eta$. Now, \eqref{eqn:countingedges} yields that $T_1/\abs{\mathcal{I}_n}$ converges to $\eta$ in probability, and the rate of convergence is of order $\exp(-\Omega(n^2))$, which in turn implies convergence almost surely.
\end{proof}

In a similar vein, we may now prove Theorem \ref{theorem:ERGMtriangleLpandTails}. To this end, we shall express the number of triangles as a linear combination of polynomials of the type $f_{d,A}$.

\begin{proof}[Proof of Theorem \ref{theorem:ERGMtriangleLpandTails}]
  For the proof fix $n \in \IN$ and let $X \sim \mu_{\bb}$. Moreover, we will write $\mu_{\Delta} \coloneqq \IE X_{e_1 e_2 e_3}, \tilde{\mu}_{\Delta} \coloneqq \tilde{X}_{e_1 e_2 e_3}$ for some triangle $\{e_1, e_2, e_3\}$, $\mu_2 \coloneqq \IE X_{e_1 e_2}, \tilde{\mu}_2 \coloneqq \IE \tilde{X}_{e_1 e_2}$ for two neighboring edges $e_1 \neq e_2$, $\mu_1 \coloneqq \IE X_e$ and lastly $(e_1, e_2, e_3) \in \mathcal{T}_n$ to indicate that $\{e_1, e_2, e_3\} \in \mathcal{T}_n$ and $e_1 < e_2 < e_3$ (with some fixed partial ordering of the edges). 
  
  Now it is not hard to verify that we can decompose $T_3(X) - \IE T_3(X)$ as
  \begin{align}  \label{eqn:T3decomposition}
    T_3(X) - \IE T_3(X) & = f_3(X) + \mu_1 f_2(X) + (n-2) \mu_2 f_1(X),
  \end{align}
  using the auxiliary functions
  \begin{align*}
  f_1(X) &\coloneqq \sum_{e \in \mathcal{I}_n} \tilde{X}_e, \qquad f_2(X)        \coloneqq \sum_{\substack{e_1 < e_2 \\ e_1 \cap e_2 \neq \emptyset}} \left(\tilde{X}_{e_1 e_2} - \tilde{\mu}_{e_1 e_2} \right), \\
    f_3(X)         & \coloneqq \sum_{(e_1, e_2, e_3) \in \mathcal{T}_n} \left( \tilde{X}_{e_1 e_2 e_3} - \tilde{\mu}_{\Delta} - \tilde{X}_{e_1} \tilde{\mu}_2 - \tilde{X}_{e_2} \tilde{\mu}_2 - \tilde{X}_{e_3} \tilde{\mu}_2 \right).
  \end{align*}

  Hence, after symmetrization of the sum, the triangle count is the sum of three terms $f_{d,A}$ for different tensors $A_3, A_2, A_1$ given by $(A_3)_{efg} = 1/6 \cdot \eins_{\{e,f,g\} \in \mathcal{T}_n}, (A_2)_{ef} = \frac{\mu_1}{2} \eins_{e \cap f \neq \emptyset}$ and $(A_1)_e = (n-2)\mu_2$. An easy counting argument shows $\norm{A_3}_2 \sim n^{3/2}/6$, $\norm{A_2}_2 \sim \mu_1 n^{3/2}/2$ and $\norm{A_1}_2 \sim \mu_2 n^2/\sqrt{2}$.

  An application of Theorem \ref{theorem:fdPolynomials} yields
  \begin{align*}
    &\norm{T_3(X) - \IE T_3(X)}_p \le (\sigma^2 \norm{A_3}_2^{2/3} p)^{3/2} + (\sigma^2 \norm{A_2}_2 p) + (\sigma^2 \norm{A_1}_2^2 p)^{1/2} \\
    &\norm{T_3(X) - \IE T_3(X) - (n-2)\mu_2 f_1(X)}_p \le (\sigma^2 \norm{A_3}_2^{2/3} p)^{3/2} + (\sigma^2 \norm{A_2}_2 p).
  \end{align*}
  The assertion now follows as in the proof of Theorem \ref{theorem:com}.
\end{proof}

\begin{rem*}
  In the case of the Erd{\"o}s--R{\'e}nyi model, the decomposition \eqref{eqn:T3decomposition} is simply the Hoeffding decomposition, with $(n-2)\mu_2 f_1(X)$ as the first order and $\mu_1 f_2(X)$ as the second order Hoeffding term, which also coincides with the decomposition of the function $T_3(X)$ in $L^2$ in terms of the orthonormal basis $(f_S(X))_{S \subset \mathcal{I}_n}, f_S(X) = (p(1-p))^{-\abs{S}/2} \prod_{s \in S} (X_s - p)$, see also Section \ref{subsection:ER}.
\end{rem*}

\begin{proof}[Proof of Corollary \ref{corollary:CLTERGM}]
  Using \eqref{eqn:T3minusf1fluctuations}, it can be shown that for any $t > 0$
  \begin{align}  \label{eqn:T3Minusf1tozero}
    \mu_{\bb}\left( \abs*{\frac{T_3 - \IE_\mu T_3 - (n-2)\mu_2 f_1}{(n-2)\mu_2 \sqrt{\binom{n}{2}}}} \ge t \right) \to 0 \text{ for } n \to \infty,
  \end{align}
  and thus
  \[
    \frac{T_3 - \mu_{\bb}(T_3)}{(n-2)\mu_2 \sqrt{\binom{n}{2}}} = \frac{T_3 - \mu_{\bb}(T_3) - (n-2)\mu_2 f_1}{(n-2)\mu_2 \sqrt{\binom{n}{2}}} + \frac{1}{\sqrt{\binom{n}{2}}} f_1 \Rightarrow \mathcal{N}(0,\sigma^2)
  \]
  by \cite[Theorem 3.1]{Bil68} and the assumption.
\end{proof}

\begin{rem*}
  Actually equation \eqref{eqn:T3Minusf1tozero} can be quantified; by \eqref{eqn:T3minusf1fluctuations}, the convergence to $0$ is of the order $\exp(-\Omega(n^{1/3}))$, which also implies
  \[
    \bigg((n-2) \mu_2 \binom{n}{2}^{1/2}\bigg)^{-1}\left( T_3 - \IE_\mu T_3 - (n-2) \mu_2 f_1 \right) \to 0 \text{ almost surely}.
  \]
\end{rem*}

\begin{proof}[Proof of Proposition \ref{proposition:WassersteinDistanceTriangles}]
	By the triangle inequality for $d_W$ it suffices to prove that $d_W(\tilde{T_3}, \tilde{L}) \le Cn^{-1/2}.$
	For any $1$-Lipschitz function we have due to $\abs{f(x) - f(y)} \le \abs{x-y}$, Theorem \ref{theorem:ERGMtriangleLpandTails} and a change of variables $s = n^{1/2} t \mu_2$
	\begin{align*}
		\Big \lvert \IE_{\mu_{\bb}} f(\tilde{T_3}) - \IE_{\mu_{\bb}} f(\tilde{L}) \Big \rvert 
		&\le \int_0^\infty \mu_{\bb}\Big( \abs{T_3(x) - (n-2) \mu_2 \sum_{e \in \mathcal{I}_n} \tilde{x}_e} \ge (n-2)\mu_2 \sqrt{\binom{n}{2}} t \Big) dt \\
		&\le 2 \int_0^\infty \exp \Big( - \frac{1}{C} \min\Big( (\mu_2 n^{1/2} t)^{2/3}, \mu_2 n^{1/2} t \Big)\Big) dt \\
		&\le 2 n^{-1/2} \mu_2^{-1} \int_0^\infty \exp \Big( - \frac{1}{C} \min(s^{2/3}, s) \Big) ds.
	\end{align*}
  Taking the supremum over all $f \in \mathrm{Lip}_1$ finishes the proof.
\end{proof}

\subsection{Central limit theorems for subgraph counts in Erd{\"o}s--R{\'e}nyi graphs}    \label{subsection:ER}
For the Erd{\"o}s--R{\'e}nyi model and for $p \in (0,1)$ we define $\sigma^2(p) \coloneqq \frac{\log(1-p) - \log(p)}{1-2p}$. Since $\mu_{n,p}$ is a product measure on $\{0,1\}^{n(n-1)/2}$, by the tensorization property and \cite[Theorem A.2]{DSC96} we have
\begin{align}      \label{eqn:LSIpartialER}
  \Ent_{\mu_{n,p}}(f^2) \le \sigma^2(p) \int \sum_{i \neq j} \Var_{\mu_{ij}}(f) d\mu_n = \sigma^2(p) \int \abs{\partial f}^2 d\mu_{n,p},
\end{align}
i.\,e. $\mu_{n,p}$ satisfies an $\mathrm{LSI}(\sigma^2(p))$. Note that as $p \to 0$ we have $\sigma^2(p) \sim \log(1/p),$ i.e. the logarithmic Sobolev constant tends to infinity, however at a logarithmic scale.

\begin{proof}[Proof of Theorem \ref{theorem:ERCLTallgraphs}]
Let $a \coloneqq (\abs{\mathrm{Aut}(G)})^{-1}$. We make use of the $L^2(\mu_{n,p})$ (or Hoeffding) decomposition of $T_G$ with respect to the orthonormal basis $(f_S)_{S \subset \mathcal{I}_n}$ given by $
    f_S = (p(1-p))^{-\abs{S}/2} \prod_{s \in S} (X_s - p),$ i.e.
  \begin{equation*}
    T_G = \sum_{S \subset \mathcal{I}_n} \skal{T_G, f_S} f_S = \sum_{k = 0}^{\abs{E}} \sum_{\substack{S \subset \mathcal{I}_n \\ \abs{S} = k}} \skal{T_G, f_S} f_S \eqqcolon \sum_{k = 0}^{\abs{E}} T_k = \mu_{n,p}(T_G) + \sum_{k = 1}^{\abs{E}} T_k.
  \end{equation*}
  For arbitrary $k \in \{1, \ldots, \abs{E}\}$ we obtain from the representation \eqref{eqn:T_G}
  \[
    T_k 
    = a(p(1-p))^{-k} \sum_{\{f_1, \ldots, f_k\}} \sum_{f: V \to [n] \text{ inj}} \skal{\prod_{e \in E} X_{f(e)}, \tilde{X}_{f_1 \cdots f_k} } \tilde{X}_{f_1 \cdots f_k}.
  \]
  For fixed $f_1, \ldots, f_k$ the scalar product is zero unless the injection uses all edges $f_1, \ldots, f_k$, giving
  \begin{align*}
    T_k & = a (p(1-p))^{-k} \sum_{\{f_1, \ldots, f_k\}} \tilde{X}_{f_1 \cdots f_k} \sum_{f: V \to [n] \text{ inj. uses edges } f_1, \ldots, f_k} p^{\abs{E}-k} (p(1-p))^{k} \\
        & = a p^{\abs{E}-k} \sum_{\{f_1 \ldots f_k\}} \tilde{X}_{f_1 \cdots f_k} N_G(f_1, \ldots, f_k),
  \end{align*}
  where $N_G(f_1, \ldots, f_k) \coloneqq \abs{\{f: V \to [n] \text{ inj. }: f \text{ uses edges } f_1, \ldots, f_k \}}$. Especially we have 
  \[
    T_1 = \sum_{e} \skal{T_G,f_{\{e\}}} f_{\{e\}} = 2a\abs{E} p^{\abs{E}-1} (n-2)_{\abs{V}-2} \sum_{e} \tilde{X}_{e}.
  \]
  We now claim that
  \[
    \frac{\sum_{k = 2}^{\abs{E}} T_k}{2a \abs{E}(n-2)_{\abs{V}-2}p^{\abs{E}-1}\sqrt{\binom{n}{2}p(1-p)}} \to 0 \text{ in probability},
  \]
  from which the result immediately follows, since the normalized first order Hoeffding term converges weakly to a standard normal distribution by the central limit theorem for i.i.d. random variables. \par
  Let us further split the $k$-th Hoeffding term. Denote by $\alpha(f_1, \ldots, f_k)$ the number of vertices that are used in the graph induced by the edge set $\{f_1, \ldots, f_k\}$ and let $\alpha_k$ denote the minimal number of vertices in a subgraph of $G$ with $k$ edges. Clearly, $\alpha(f_1, \ldots, f_k) \in \{ \alpha_k, \ldots, 2k \wedge \abs{V} \}$, where $a \wedge b \coloneqq \min(a,b)$. This results in the decomposition
  \begin{align*}
    T_k & = \sum_{\alpha = \alpha_k}^{\abs{V} \wedge 2k} T_{k,\alpha}\qquad \text{and}\qquad \tilde{T}_G \coloneqq T_G - \mu_{n,p}(T_G) - f_1 = \sum_{k = 2}^{\abs{E}} \sum_{\alpha = \alpha_k}^{\abs{V} \wedge 2k} T_{k,\alpha}.
  \end{align*}
  Using the $L^q(\mu_{n,p})$ estimates from Proposition \ref{theorem:fdPolynomials} yields for all $q \ge 2$
  \begin{align*}
    \norm{\tilde{T}_G}_q & \le \sum_{k = 2}^{\abs{E}} \sum_{\alpha = \alpha_k}^{\abs{V} \wedge 2k} \norm{T_{k,\alpha}}_q \le \sum_{k = 2}^{\abs{E}} \sum_{\alpha = \alpha_k}^{\abs{V} \wedge 2k} \left( \sigma^2(p) p^{\abs{E}-k} \norm{A^{(k,\alpha)}}_2^{k/2} q \right)^{2/k}.
  \end{align*}
  Let $c(n,p) \coloneqq (\abs{E}(n-2)_{\abs{V}-2}p^{\abs{E}-1}\sqrt{\binom{n}{2}p(1-p)})$. The $L^q$ estimates can be used to show the multilevel concentration inequality (as in the proof of Theorem \ref{theorem:com})
  \begin{align*}
    \mu_{n,p}\left( c(n,p)^{-1} \abs{\tilde{T}_G} \ge t \right) \le 2 \exp \left( - \frac{1}{C \sigma^2(p)} \min_{2 \le k \le \abs{E}} \min_{\alpha_k \le \alpha \le 2k \wedge\abs{V}} t^{2/k} h_{n,k,\alpha}^{2/k} \right)
  \end{align*}
  with $h_{n,k,\alpha} \coloneqq \frac{c(n,p)}{p^{\abs{E}-k} \norm{A^{(k,\alpha)}}}_2$, and it remains to prove that
  $
    \min_{k = 2, \ldots, \abs{E}} h_{n,k,\alpha} \to \infty
  $.
  To this end, we estimate $\norm{A^{(k,\alpha)}}_2$ from above as follows. If there are $\alpha$ vertices used by the edges $f_1, \ldots, f_k$, there are at most $n^{\abs{V}-\alpha}$ ways to have an injection of $V$ into $[n]$ using the edges $f_1, \ldots, f_k$, and there are at most $n^\alpha$ such combinations of $f_1, \ldots, f_k$, and thus
  \[
    \norm{A^{(k,\alpha)}}_2 \le n^{\abs{V}-\alpha} n^{\alpha/2} = n^{\abs{V} - \alpha/2}.
  \]
  For $k \ge 2$ this gives for some constant $c(G)$ depending on the graph $G$
  \begin{align*}
    h_{n,k,\alpha} &= \frac{c(n,p)}{p^{\abs{E}-k} \norm{A^{(k,\alpha)}}_2} \ge c(G) \frac{n^{\abs{V}-1}p^{\abs{E}-1/2}}{n^{\abs{V}-\alpha/2} p^{\abs{E}-k}} \ge c(G) \left( n^{\alpha-2} p^{2k-1} \right)^{1/2} 
  \end{align*}
  leading to the condition $\min_{k = 2, \ldots, \abs{E}} \min_{\alpha \in \{\alpha_k, \ldots, 2k \wedge \abs{V}\}} n p^{(2k-1)/(\alpha-2)} = np^{d'(G)} \to \infty$.
\end{proof}

\subsection{Proofs of the general results}
\begin{proof}[Proof of Theorem \ref{theorem:com}]
  First, note that since $\mu$ satisfies a $\partial\mathrm{-LSI}(\sigma^2)$, by \cite[Proposition 2.4]{GSS18} we obtain for any $p \ge 2$
  \begin{align}  \label{eqn:momentinequality2}
    \norm{f - \IE_\mu f}_p \le (\sigma^2 p)^{1/2} \norm{\mathfrak{h} f}_p.
  \end{align}
  Next we iterate \eqref{eqn:momentinequality2} using $\abs{\mathfrak{h}\abs{\mathfrak{h}^{(d-1)} f}} \le \abs{\mathfrak{h}^{(d)} f}$ (cf. \cite[Lemma 2.3]{GSS18}), leading to
  \begin{align}    \label{eqn:Lpinequalities}
    \norm{f- \IE_\mu f}_p \le \sum_{k = 1}^{d-1} (\sigma^2 p)^{k/2} \norm{\mathfrak{h}^{(k)} f}_2 + (\sigma^2p)^{d/2} \norm{\mathfrak{h}^{(d)} f}_p.
  \end{align}
  To prove the multilevel concentration bounds \eqref{eqn:multilevelconcentration}, we use methods outlined in \cite[Theorem 7]{Ad06} and \cite[Theorem 3.3]{AW15}. To sketch the method in a slightly more general situation, assume that there are constants $C_1, \ldots, C_d \ge 0$ such that for any $p \ge 2$
  \[
    \norm{f-\IE_\mu f}_p \le \sum_{k = 1}^{d} (C_k p)^{k/2}.
  \]
  Let $N \coloneqq \abs{\{ n : C_n > 0\}}$ and $r \coloneqq \min \{ k \in \{1,\ldots,d\} : C_k > 0 \}$. By Chebyshev's inequality we have for any $p \ge 1$
  \begin{equation}  \label{eqn:ChebyshevWithLp}
    \mu ( \abs{f-\IE_\mu f} \ge e \norm{f-\IE_\mu f}_p) \le \exp(-p).
  \end{equation}
  Now consider the function
  \[
    \eta_f(t) \coloneqq \min \left\{ \frac{t^{2/k}}{C_k} \colon k = 1, \ldots, d \right\},
  \]
  with $\frac{x}{0}$ being understood as $\infty$. If we assume $\eta_f(t) \ge 2$, we can estimate $
    e \norm{f - \IE_\mu f}_{\eta_f(t)} \le e \sum_{k = 1}^{d} \eins_{C_k \neq 0} t = Net,$
  so that an application of equation \eqref{eqn:ChebyshevWithLp} to $p = \eta_f(t)$ yields
  \[
    \mu(\abs{f- \IE_\mu f} \ge (Ne)t) \le \mu(\abs{f- \IE_\mu f} \ge e\norm{f - \IE_\mu f}_{\eta_f(t)}) \le \exp\left( - \eta_f(t) \right).
  \]
  Combining it with the trivial estimate $\mu(\cdot) \le 1$ (in the case $\eta_f(t) \le 2$) gives
  \[
    \mu(\abs{f- \IE_\mu f} \ge (Ne)t) \le e^2 \exp(-\eta_f(t)).
  \]
  To remove the $Ne$ factor, rescaling the function by $Ne$ and using the estimate $\eta_{(Ne)f}(t) \ge \frac{\eta_{f}(t)}{(Ne)^{2/r}}$ yields
  \[
    \mu(\abs{f - \IE_\mu f} \ge t) \le e^2 \exp \left( - \frac{1}{(Ne)^{2/r}} \eta_f(t) \right).
  \]
  In a last step, note that $e^2\exp(- (Ne)^{-2/r} \eta(t)) \le 2 \exp( - \log(2) (Ne)^{-2/r} \eta(t)/2)$ in the nontrivial regime $(Ne)^{-2/r} \eta(t) \ge 2$, so that
  \[
    \mu(\abs{f - \IE_\mu f} \ge t) \le 2 \exp \left( - \frac{\log(2)}{2(Ne)^{2/r}} \eta_f(t) \right).
  \]
\end{proof}

To prove Theorem \ref{theorem:fdPolynomials}, let us introduce another notation. For any indices $l_1,\ldots,l_s \in \mathcal{I}$ and $s$ distinct indices $k_1,\ldots,k_s \in \{1,\ldots,d \}$ let $A^{k_1=l_1,\ldots,k_s=l_s}$ be the $(d-s)$-tensor with fixed entries $k_i = l_i$ for all $i = 1,\ldots,s$. For example, if $A = (A_{ijkl})$ is a $4$-tensor, $A^{2=j, 3=i}$ is the $2$-tensor given by $A^{2=j,3=i}_{kl} = A_{kjil}$. Clearly, if $A$ is symmetric, then $A^{k_1 = l_1, \ldots, k_s = l_s}$ is symmetric; and if $A$ has a vanishing diagonal, then so has $A^{k_1 = l_1, \ldots, k_s = l_s}$.

\begin{proof}[Proof of Theorem \ref{theorem:fdPolynomials}]
  To see that $\IE_\mu f_{d,A} = 0$ fix $i_1,\ldots,i_d$ and an arbitrary decomposition $P \in \mathcal{P}(\{i_1,\ldots, i_d\})$. If $N(P) = 1$, then $g_P$ has mean zero. On the other hand, if $N(P) \ge 2$, then $P = \{\{i_1\}, \ldots, \{i_{N(P)}\}, I_1, \ldots, I_l \}$ ($l \ge 0$), but the set $\tilde{P} = \{ \{i_1,\ldots, i_{N(P)} \}, I_1, \ldots, I_l \}$ is also a valid partition and $g_{\tilde{P}} = \IE_\mu g_P$. As a consequence, $\IE_\mu f_{d,A} = 0$. \par
  For any $l \in \mathcal{I}$ write $T_l$ for the formal operator that replaces $x_l$ by $\hat{x}_l$, and use the short-hand notation $x_I \coloneqq (x_i)_{i \in I}$ for any $I \subset \mathcal{I}$. We shall make use of the inequality
  \begin{align*}
    \mathfrak{h}_l(f_{d,A}) & = \sup_{x_l, \hat{x}_l} \Big \lvert\sum_{I = (i_1,\ldots,i_d)} A_I \sum_{P \in \mathcal{P}(I)} (-1)^{M(P)}  \left( g_P(x_I)) - g_P(T_l(x_I)) \right) \Big \rvert                   \\
    & = \sup_{x_l, \hat{x}_l} \Big \lvert(f(x_l) - f(\hat{x}_l))\sum_{k = 1}^d \sum_{I = (i_1,\ldots, i_{d-1})} A^{k = l}_I \sum_{P \in \mathcal{P}(I)} (-1)^{M(P)} g_P(x_I) \Big \rvert \\
    & \le c \Big \lvert\sum_{k = 1}^d \sum_{I = (i_1,\ldots, i_{d-1})} A^{k=l}_I \sum_{P \in \mathcal{P}(I)} (-1)^{M(P)} g_P(x_I)\Big \rvert \\
    & = c \Big \lvert \sum_{k= 1}^d f_{d-1, A^{k=l}} \Big \lvert.
  \end{align*}
  Here, the second equality follows from the fact that $T_l(x_{i_1},\ldots,x_{i_d}) = (x_{i_1},\ldots, x_{i_d})$ unless $i_j = l$ for some $j$ and the definition of $g_P$, and the inequality in the third line is a consequence of the assumptions. \par
  We can assume $c = 1$, since the general case follows by rescaling $f$ by $c^{-1}$. First, by the $\partial-\mathrm{LSI}(\sigma^2)$ property we have
  \[
    \norm{f_{d,A}}_p^2 \le \norm{f_{d,A}}_2^2 + \sigma^2(p-2) \norm{\mathfrak{h}(f_{d,A})}_p^2.
  \]
  Using the Poincar{\'e} inequality with respect to $\mathfrak{h}$ \eqref{eqn:generalPI} gives
  \begin{align*}
    \norm{f_{d,A}}_2^2 &\le \sigma^2 \sum_{l_1} \mu\left( (\mathfrak{h}_{l_1}f_{d,A})^2 \right) \le \sigma^2 \sum_{l_1} \mu \left( (\tilde{\mathfrak{h}_{l_1}} f_{d,A})^2 \right) = \sigma^2 \norm{\tilde{\mathfrak{h}}f_{d,A}}_2^2 \le \sigma^2 \norm{\tilde{\mathfrak{h}}f_{d,A}}_p^2,
  \end{align*}
  where $\tilde{\mathfrak{h}}_l$ replaces $\sup_{x_l, \hat{x}_l} \abs{f(x_l) - f(\hat{x}_l)}$ by $1$. Clearly, since $\mathfrak{h}_l f_{d,A} \le \tilde{\mathfrak{h}_l} f_{d,A}$ pointwise, the $L^p$-norms can be estimated as well, resulting in $
    \norm{f_{d,A}}_p^2 \le \sigma^2(p-1) \norm{\tilde{\mathfrak{h}} f_{d,A}}_p^2.$
  We have
  \[
    \tilde{\mathfrak{h}}_{l_1} f_{d,A} = \abs*{\sum_{k_1 = 1}^d \sum_{I = (i_1,\ldots, i_{d-1})} A_I^{k_1 = l_1} \sum_{P \in \mathcal{P}(I)} (-1)^{M(P)} g_P},
  \]
  which itself is the absolute value of a sum of centered random variables, so that the process can be iterated; in each step, the Poincar{\'e} inequality \eqref{eqn:generalPI} can be used and
  \[
    \tilde{\mathfrak{h}}_{l_1} \cdots \tilde{\mathfrak{h}}_{l_s} f_{d,A} = \abs*{\sum_{k_1 = 1}^d \cdots \sum_{k_s = 1}^{d-s} \sum_{I = (i_1,\ldots, i_{d-s})} A^{k_1 = l_1, \ldots, k_s = l_s}_I \sum_{P \in \mathcal{P}(I)} (-1)^{M(P)} g_P}.
  \]
  Thus, using the inequality $\tilde{h}\abs{\tilde{h}^{(d)} f} \le \abs{\tilde{h}^{(d+1)} f}$ and taking the square root yields
  \[
    \norm{f_{d,A}}_p \le (\sigma^2p)^{d/2} \norm{A}_2.
  \]
  The multilevel concentration follows as in the proof of Theorem \ref{theorem:com}.
\end{proof}

\printbibliography

\end{document}